\documentclass[12pt]{amsart}
\usepackage{amsmath, amssymb, mathtools, a4wide, mathrsfs, bm, mathptmx}
\usepackage{hyperref}
\usepackage{paralist}

\usepackage[T1]{fontenc}

\numberwithin{equation}{section}

\newenvironment{psm}
  {\left(\begin{smallmatrix}}
  {\end{smallmatrix}\right)}

\numberwithin{equation}{section}

\theoremstyle{plain}
\newtheorem{thm}{Theorem}[section]
\newtheorem{lem}[thm]{Lemma}
\newtheorem{prop}[thm]{Proposition}
\newtheorem{cor}[thm]{Corollary}
\newcommand{\thmref}[1]{Theorem~\ref{#1}}
\newcommand{\lemref}[1]{Lemma~\ref{#1}}
\newcommand{\propref}[1]{Proposition~\ref{#1}}
\newcommand{\corref}[1]{Corollary~\ref{#1}}

\theoremstyle{definition}
\newtheorem{rmk}[thm]{Remark}
\newtheorem{defi}[thm]{Definition}

\newcommand{\rmkref}[1]{Remark~\ref{#1}}
\newcommand{\defiref}[1]{Definition~\ref{#1}}

\newcommand{\re}{\mrm{Re}}
\newcommand{\im}{\mrm{Im}}
\newcommand{\mbb}{\mathbb}
\newcommand{\x}{\textbf}
\newcommand{\mf}{\mathbf}
\newcommand{\q}{\quad}

\newcommand{\mc}{\mathcal}

\newcommand{\mrm}{\mathrm}
\newcommand{\sltwo}{\mrm{SL}(2, \mf Z)}

\newcommand{\symtwor}{\mrm{Sym}_2(\mf R)}

\newcommand{\sptwo}{\mrm{Sp}(2, \mf Z)}
\newcommand{\sptwor}{\mrm{Sp}(2, \mf R)}
\newcommand{\sltwor}{\mrm{SL}(2, \mf R)}

\newcommand{\spnr}{\mrm{Sp}(n, \mf R)}
\newcommand{\gltwo}{\mrm{GL}(2, \mf Z)}

\newcommand{\bphi}{\bm{\phi}}
\newcommand{\ut}{\underset}

\newcommand{\symnr}{\mrm{Sym}_n(\mf R)}
\newcommand{\di}{\frac{\partial}{\partial z_{2}}} 
\newcommand{\dib}{\frac{\partial}{\partial \bar{z_{2}}}} 

\newcommand{\tr}{\mathrm{tr}}

\begin{document}

\title[Growth of Petersson norms]{Petersson norms of not necessarily cuspidal Jacobi modular forms and applications}

\author{Siegfried B\"ocherer}
\address{Institut f\"ur Mathematik\\
Universit\"at Mannheim\\
68131 Mannheim (Germany).}
\email{boecherer@math.uni-mannheim.de}

\author{Soumya Das}
\address{Department of Mathematics\\ 
Indian Institute of Science\\ 
Bangalore -- 560012, India.}
\email{somu@math.iisc.ernet.in, soumya.u2k@gmail.com}

\date{}

\dedicatory{In the memory of Prof. H. Klingen (1927--2017)}

\begin{abstract}
We extend the usual notion of Petersson inner product on the space of cuspidal Jacobi forms to include non-cuspidal forms as well. This is done by examining carefully the relation between certain ``growth-killing" invariant differential operators on $\mf H_2$ and those on $\mf{H}_1 \times \mf{C}$ (here $\mf H_n$ denotes the Siegel upper half space of degree $n$). As applications, we can understand better the growth of Petersson norms of Fourier Jacobi coefficients of Klingen Eisenstein series, which in turn has applications to finer issues about representation numbers of quadratic forms; and as a by-product we also show that \textit{any} Siegel modular form of degree $2$ is determined by its `fundamental' Fourier coefficients.
\end{abstract}
\maketitle 

\section{Introduction}
Non-cuspidal elliptic modular forms decompose into an Eisenstein series
and a cusp form. This gives at the same time a decomposition of
its Fourier coefficients into a dominant term 
(coming from the Eisenstein series)
and an error term (coming from the cuspidal part).
For Siegel modular forms of higher degrees the main obstacle to clean asymptotic properties
of the Fourier coefficients are the so-called Klingen-Eisenstein
series attached to cusp forms of lower degree. Their Fourier coefficients
grow in a similar way as the Fourier coefficients of Siegel Eisenstein series (which are natural candidates for the dominant term), 
as long as we consider them indexed by matrices in certain 
subsets of half-integral symmetric and positive definite matrices. We refer the reader to \cite{BR, KiTata} for a variety of results in this direction. 

A first attempt to a better understanding of this phenomenon was made
in \cite{BV} for the special case of degree $2$. The key observation
was to use the Fourier-Jacobi coefficients $\phi_m$ of the Klingen 
Eisenstein series, and to consider its decomposition into a Jacobi-Eisenstein part and a cuspidal
part:
\[ \phi_m={\mathcal E}_{k,m}+\phi_m^o.\]
Vaguely speaking, the cuspidal part $\phi_m^\circ$ behaves (almost) 
like the Fourier-Jacobi coefficient of a Siegel cusp form, whereas the
${\mathcal E}_{k,m}$'s are responsible for the dominating part of the Fourier
coefficients of the Klingen Eisenstein series.
These properties were shown in \cite{BV} for special types of exponent matrices $T$, e.g., those $T$ for which $-\det(2T)$ is a fundamental discriminant. 
The basic tool in \cite{BV} was to identify the Fourier coefficients of the ${\mathcal E}_{k,m}$ with subseries of the infinite series giving the Fourier coefficients of the Klingen-Eisenstein series. This interplay was recently worked out in complete generality in the Ph.D. thesis of T. Paul \cite{P, PEng}.

One of the main purposes of the present paper is to get a 
better understanding of the growth 
properties of those cuspidal parts $\phi_m^o$, in particular
their Petersson norms; from this one then gets
growth properties also for the Fourier coefficients. This is obtained in the form of an asymptotic formula, in \thmref{genasy}. To do this we take a detour (mainly because the Dirichlet series $\sum_{m \geq 1}\langle \phi_m^\circ,\phi_m^\circ \rangle m^{-s}$ does not seem to have good analytic properties, see \rmkref{detour}), which is at the same time the 
second main topic of our paper. Namely using certain
``growth killing" invariant differential operators, 
we define an extension $\{ \, ,\}$ of the classical Petersson product $\langle \, , \rangle$ to the full space
of Jacobi forms. The idea is that we are able to estimate $\{\phi_m,\phi_m\}$ 
first and compute $\{{\mathcal E}_{k,m},{\mathcal E}_{k,m}\}$ 
explicitly; thereby obtaining a bound for $\{\phi_m^\circ,\phi_m^\circ \}$.
Note that this is proportional 
to the square of the standard Petersson norm for $\phi_m^\circ$, because
$\phi_m^\circ$ is cuspidal. 

Before discussing the content of the paper in more detail, let us indicate a few applications of our construction and ideas involved in this paper. \textsl{\textbf{Throughout this paper, we assume that $k$ is even.}} One of the reasons is that there is no reasonable way of introducing Eisenstein series for odd weights for the Siegel modular group (cf. \cite[p.~63]{Fr}).

\subsection*{Applications:}  Apart from the intrinsic interest in
the extended Petersson norm, we can give the following applications.

\noindent \textsl{(i)} Firstly, we can prove a version of an asymptotic formula for the representation number of a given binary even integral quadratic form by an even unimodular form, which is more refined than what was known before in previous works, eg., \cite{BV}, \cite[\S~4]{Ki1}. In particular, we prove the following. Let $T$ be a positive definite binary quadratic form and $\min(T)$ be its minimum. Suppose that $\min(T)$ is represented by a positive definite even unimodular quadratic form $S$ in $2k$ variables ($k \geq 4$). Then for any $\epsilon>0$, the number of representations $A(S,T)$ of $T$ by $S$ is at least $C \cdot \det(T)^{k-3/2-\epsilon}$, where $C$ is a constant depending only on $k,S$. See \corref{repno-nonzero}. We refer the reader in particular to \cite[Rmk.~3, Thm.~IV]{BR} and the work of Kitaoka \cite{KiTata} for more details. The main point is that our results are uniform in $\det(T)$ and do not depend on conditions like $\min(T) \to \infty$, which seem to be present in all the earlier works.

\noindent \textsl{(ii)}  Secondly, we answer a question raised in the paper by A. Saha \cite{saha} affirmatively by showing that if $F$ is a non-zero Siegel modular form of degree $2$, then it has infinitely many non-zero `fundamental' Fourier coefficients, i.e., $a(F,T) \neq 0$ with $-\det(2T)$ a fundamental discriminant. This follows from the finer asymptotics in \thmref{genasy} along with a certain `$\Omega$'-result for Fourier coefficients of elliptic cusp forms. Previously this was only known from \cite{saha} in the case of cusp forms. See \propref{saha-all}.

\noindent \textsl{(iii)}  As a final application of the techniques developed here we generalize
the Dirichlet series introduced by Kohnen-Skoruppa in \cite{KS}
to not necessarily cuspidal Siegel modular forms by considering 
$\sum_{m \geq 1} \{\phi_m,\phi_m\} m^{-s}$; and show that this series has essentially the same properties as in the cuspidal case. Note however that unlike the case in \cite{KS}, we cannot use Landau's theorem to study the growth of $\{ \, , \}$, as this inner product may not be positive definite.

\subsection*{Discussion of the main topics and the structure of the paper:}  Invariant differential operators are quite 
complicated objects for Jacobi groups, e.g.
the ring of such operators is not commutative, 
see \cite{Ochiai} for 
the general picture and \cite{BS, Pitale} for details in the case of the classical Jacobi group.
For our purpose, the
situation is even more complicated because strictly speaking, 
we are not so much interested
in the intrinsic theory of such operators but in their relation  
to growth killing differential operators for Siegel 
modular forms, because after all the Jacobi forms which 
we consider in the applications arise from Siegel modular forms.
 
We rely on extensive computation with the invariant differential operators, which is treated in an appendix, to get hold of some crucial identities \eqref{DJ0}, \eqref{DJ1}, \eqref{DJ2} needed for further considerations. A more conceptual understanding here is desirable, nevertheless this paper could be a starting point for further such investigations. To get hold of `growth killing' operators for functions on $\mf H_1 \times \mf C$ seem to be rather non-trivial, and we approach this by taking cue from the corresponding known results for the Siegel upper half space $\mf H_2$, as shown by Maa{\ss} \cite{Ma}. Namely we extend a function on $\phi \colon \mf H_1 \times \mf C$ to $\mf H_2$ in natural ways (see \eqref{h}, \eqref{H}) so that they are functions invariant under the Klingen parabolic subgroup $C_{2,1}(\mf Z)$. Then we decompose the action of a growth killing operator $\mc M$ on $\mf H_2$ while restricting its automorphy to $C_{2,1}(\mf Z)$, to get certain linear combination of `growth killing' operators on $\mf H_1 \times \mf C$. The calculations in the appendix then allow us to write these operators in terms of the generators of the ring of invariant Jacobi differential operators. Once this is done, we can define the extended inner product, see \defiref{Pdef}.

Then in \propref{j2h}, we relate the new inner product $\{ \, , \}$ with that for the space of modular forms of half-integral weights on $\mf H_1$; this allows us to show easily that $\{ \, , \}$ is an extension of $\langle \, , \rangle$. We apply these in section~\ref{phi0} to Fourier Jacobi coefficients of Siegel modular forms and derive a bound $\{ \phi^\circ_m, \phi^\circ_m \} \ll m^k$ for the cuspidal parts of these objects. This in turn gives better bounds (cf. \cite{BV}) for the Fourier coefficients of these $\phi^\circ_m$. See \thmref{phi0-bd}. To do this, we need bounds for the Petersson norms of the corresponding Eisenstein parts. Sections~\ref{ekm-bd}, \ref{S2J} are devoted to deal with these. The main point here is that one is able to express $\mc E_{k,m}$ explicitly in terms of the Fourier-Jacobi coefficients of the Siegel Eisenstein series of degree $2$ 
and 
ultimately (see theorem 4.4)
in the form $$\mc E_{k,m} = E_{k,1} \vert_{k,m} \mc T_m$$ for some explicit Hecke operator $\mc T_m$. Here $E_{k,1}$ denotes the Jacobi Eisenstein series of weight $k$ and index $1$. This is a somewhat surprising result.

In order to consider the case of arbitrary lattices (thus not necessarily unimodular), one would have to consider Klingen-Eisenstein series of higher levels; where matters get complicated in terms of the description of their Fourier coefficients. However most of our abstract considerations concerning the restriction of differential operators from the Siegel to Jacobi spaces go through, being at the level of Lie groups. It would be interesting to generalise our formulas to this situation.

\subsection*{Acknowledgements} 
The bulk of this work was done during mutual visits by
the authors. S.B. thanks IISc Bangalore for generous hospitality and S.D. thanks the University of Mannheim for providing excellent working conditions. He also thanks IISc. Bangalore, DST (India) and UGC centre for advanced studies for financial support. The authors are extremely grateful to the anonymous referee for a meticulous checking of the paper and for several helpful comments.

\section{Notation and preliminaries}
\subsection{General notation}
\begin{inparaenum}[(1)]
\item 
For a commutative ring $R$ with $1$, we denote by $M_{m,n}( R)$ to be set of $m \times n$ matrices with coefficients in $R$. If $m=n$, we put  $M_{m,n}( R)= M_n(R)$. We denote by $\mrm{Sym}_n(\mf R)$ (resp. $\mrm{Sym}_n(R)^+$) the space of symmetric (resp. positive definite) matrices over the reals $\mf R$. Further, the $n\times n$ identity matrix over a subring of $\mf C$ is denoted by $1_n$.

The rank of a matrix $M$ is denoted by $\mrm{rank}(M) = \mrm{r}(M)$. The notation $M[N] := N'MN$ for matrices of appropriate size is used, where $N'$ denotes the transpose of $N$. 

We define the set of half-integral, symmetric, non-negative matrices, by 
\[ \Lambda_n := \{ S= (s_{i,j})  \in M (n, \tfrac{1}{2}\mf Z ) \cap \symnr \mid s_{i,i} \in \mf Z, \text{ and } S \text{ is positive semi-definite}\} \]
and denote the subset of positive definite matrices in $\Lambda_n$ by $\Lambda^+_n$.

Throughout the paper, $\varepsilon$ denotes a small positive number which may vary at different places. Moreover the symbols $A \ll_c B$ and $O_S(T)$ have their standard meaning, implying that the constants involved depend on $c$ or the set $S$. Further $A \asymp B$ means that $c_1 A \leq B \leq c_2 A$ for constants $c_1,c_2>0$.

\item
For $T$ real and $Z \in M_n(\mf C)$ we define $e(TZ) := \exp(2 \pi i \mrm{tr}(TZ))$, where $\mrm{tr}(M)$ is the trace of the matrix $M$. We denote by 
\[\mf H_n := \{ Z \in M_n(\mf C) \mid Z=Z', \mrm{Im}(Z)>0 \}, \]
the Siegel upper half-space of degree $n$. For $Z \in \mf H_n$ we usually write $Z=X+iY$, with $ X=\re(Z), Y=\im(Z) $. We would mainly need these when $n=1,2$.
In particular, we will decompose $Z\in {\mf H}_2$ as
$Z= \begin{psm} z_1 & z_2\\
z_2 & z_4\end{psm} $ with $z_j=x_j+iy_j$. 
In the context of Jacobi forms however,  we will often use
the notation $Z=\begin{psm} \tau & z\\z &\tau'\end{psm}$
with $\tau=u+iv$, $z=x+iy$, ${\tau}'=u'+iv'$.

\end{inparaenum}

\subsection{Siegel modular forms}
\begin{inparaenum}[(1)]
\item
The symplectic group $\sptwor$ acts on Siegel's half-space in the usual way by $Z \mapsto g \langle Z \rangle=(AZ+B)(CZ+D)^{-1}$ and
on functions $f:{\mf H}_2\longrightarrow {\mf C}$ by the \lq\lq stroke\rq\rq operator: for $g=\left(
\begin{smallmatrix}
A & B\\ C & D\end{smallmatrix}\right)\in \sptwor$, we define $(f\mid_k g)(Z)= j(g,Z)^{-k} f(g \langle Z \rangle) $ where $j(g,Z)=\det(CZ+D)$.

We put $\Gamma_2 := \sptwo$. A holomorphic function $f$ on
${\mf H}_2$ is called a modular form for $\Gamma_2$ of weight $k$ if it satisfies the transformation law $ \, f \mid_k\gamma= f \, $ for all $\gamma=\left(\begin{smallmatrix}
A & B \\ C & D\end{smallmatrix}\right)\in \Gamma_2) $.
We denote the space of all such functions by $M^2_k$ and
$S^2_k$ will be the subspace of cusp forms. An element $f\in M^2_k$ has a Fourier expansion
\[ f(Z) =\sum_{T\in \Lambda_2} a(f,T) \exp (\mrm{tr} \ TZ ), \] 
where $\mrm{tr \ A}$ denotes the trace of the matrix $2 \pi i A$. Then $S^2_k$ consists of those $f \in M^2_k$ for which the Fourier expansion is supported on elements of $\Lambda_2^+$. Moreover $M^1_k$ denotes the space of modular forms on $\sltwo$ of weight $k$ and $S^1_k$ the space of cusp forms therein.

\item
We next define and set up notation for the several kinds of Eisenstein series in $M^2_k$.
For $0 \leq r \leq 1$, put $C_{2,r} := \{ g \in \Gamma_2 \mid g= \left(\begin{smallmatrix}
* & * \\ 0_{2-r,2+r} & * \end{smallmatrix}\right) \}$. Then $C_{2,r}$ is a subgroup of $\Gamma_2$, usually referred to as the Siegel ($r=0$) or Klingen parabolic subgroup ($r=1$) of $\Gamma_2$.
Given a cusp form $f \in S_k^r$ (for $r=0$, we take $f=1$), the Klingen Eisenstein series attached to $f$ is defined by
\begin{equation} \label{infklin}
E_{2,r}^k(Z) = \underset{ g \in  C_{2,r} \backslash \Gamma_2}\sum \  f(g \langle Z \rangle^*) j(g,Z)^{-k}
\end{equation}
\noindent where, for a $2 \times 2$ matrix $Z \in \mf H_2$, we denote by $Z^*$ the upper left $r \times r$ block of $Z$. It is well known that $E_{2,r}^k$ converges absolutely and uniformly when $k > n+r+1$ and defines an element of $M_k^2$. When $r=0$, $E_{2,0}^k$ is nothing but the Siegel's Eisenstein series of degree $2$. 

We put $M^{2,0}_k = \mf C \cdot E_{2,0}^k$, $M^{2,1}_k = \mf C \cdot \{ \cdot E_{2,1}^k(h , \cdot) \vert h \in S^1_k \}$ and $ M_{k}^{2,2} := S_k^2$. The structure theorem says that for $k >4$ one has the decomposition 
\begin{equation} \label{mkn}
M_k^2 = \bigoplus\limits_{r=0}^2 \, M_k^{n,r}. 
\end{equation}
We refer the reader to \cite{Fr, klingen} for basic facts on the theory of Siegel modular forms.
\end{inparaenum}

\subsection{Jacobi forms} \label{j-prel}
\begin{inparaenum}[(1)]
\item
We define $G^J(\mf R) := \sltwor \ltimes \mf R^2$ and put $\Gamma^J:= \sltwo \ltimes \mf Z^2$. Let us recall the embedding of $G^J(\mf R)$ into $\sptwor$ (as sets):
\begin{equation} \label{embed}
M = \left( \begin{pmatrix} a & b \\c & d  \end{pmatrix} , \left[ \lambda, \mu \right] \right) \longmapsto  \widetilde{M}:= \begin{pmatrix}
a & 0 & b & \mu \\
\lambda' & 1 & \mu' & 0 \\
c & 0 & d & - \lambda \\
0 & 0 & 0 & 1
\end{pmatrix},  
\end{equation}
which we use in a number of occassions. The space of Jacobi forms of weight $k$ and index $m \geq 1$ is defined to be the set of holomorphic functions $\phi \colon \mf H_1 \times \mf C \rightarrow \mf C$ which are automorphic with respect to the discrete group $\Gamma^J$, i.e., if we put $g= \left( \left( \begin{smallmatrix}a&b\\c&d\end{smallmatrix} \right) , [\lambda, \mu] \right) \in \Gamma^J$, then
\begin{equation*}
\phi \vert_{k,m} g :=(c\tau+d)^{-k}e^{2\pi im\left(-\frac{c(z+\lambda\tau+\mu)^2}{c\tau +d}\right) +\lambda^2\tau + 2\lambda z}  \phi\left(\frac{a\tau+b}{c\tau+d}, \frac{z+\lambda\tau+\mu}{c\tau+d}\right) = \phi(\tau,z),
\end{equation*}
Further we demand that the Fourier expansion of $\phi$ have the shape
\[  \phi(\tau,z)=
\sum_{n, r\in \mathbf Z, \, 4mn \ge r^2 } c_\phi(n,r) e(n \tau + rz).
 \]
Moreover $\phi$ belongs to the space of cusp forms $J^{cusp}_{k,m}$ if in the above Fourier expansion only terms with $n,r$ such that $4mn>r^2$ survive.

\item
We briefly recall the theta decomposition of Jacobi forms. It is well known that the Fourier coefficients $c_\phi(n,r)$ depend only on $D:=4mn-r^2$ and $r \bmod 2m$. Sometimes, in view of this, it is more convenient to write $c_\phi(n,r) = c_\phi(D,r)$. Moreover, any such Jacobi form $\phi(\tau,z)$ can be (uniquely) written as (or we sometimes say has the theta decomposition)
\begin{equation}\label{thetad}
\phi(\tau,z) = \sum_{\mu \bmod{2m}} h_{m,\mu}(\tau)\theta_{m,\mu}(\tau,z),
\end{equation}
where we have put 
$\theta_{m,\mu}(\tau,z) =  \sum_{r \in \mathbf{Z}, \, r \equiv \mu\bmod{2m} }  e( \frac{r^2}{4m} \tau + rz)$ for the (congruent) Jacobi theta series of weight $1/2$ and index $m$; and
$h_{m,\mu}(\tau) = \sum_{n \in \mathbf{Z}, \, n\ge {\mu}^2/4m} 
c_{\phi}(n,\mu)  e{( (n-\frac {\mu^2}{4m})\tau) }$ are the so-called theta-components of $\phi$. Note that the automorphy of $\phi$ implies that $h_{m,\mu} \in M^1_{k-1/2}(\Gamma(4m))$ for all $\mu$, where $M^1_{k-1/2}(\Gamma(4m))$ denotes the space of weight $k-1/2$ modular forms on the principal congruence subgroup $\Gamma(4m)$.

\item
For $f \in M^2_k$, denote by $\phi_m(f)$ its Fourier-Jacobi coefficients, where for $Z= \left(  \begin{smallmatrix} \tau & z \\ z & \tau' \end{smallmatrix} \right)$,
\[  f(Z) =  \sum_{m=0}^\infty \phi_m(f) e(m \tau')\]
and omit the dependence on $f$ whenever convenient. Note that $\phi_m(f) \in J_{k,m}$ and is a cusp form for all $m \geq 1$ if $f$ is a cusp form.

\end{inparaenum}

\section{Growth-killing invariant differential operators on ${\mf H}_2$ and on
${\mf H_1}\times {\mf C}$} \label{grk}
\subsection{Definition and some basic properties} \label{gr-defi}
We call a $ \sptwor$-invariant differential operator
${\mathbb D}={\mathbb D}^2_k$ on ${\mf H}_2$  growth-killing (with respect to $k$),
if for all $S,T\in \symtwor$ with $S$ and $T$ positive semidefinite and $Z=X+iY\in {\mf H}_2$
we have
\begin{equation} \label{D}
{\mathbb D}\left(\det(Y)^k\cdot e^{2\pi i\sigma(S-T)X)}e^{-2\pi \sigma(S+T)Y}\right)=
0\quad\mbox{unless} \q  S+T>0 .
\end{equation}

Similarly, we call a $G^{J}({\mf R})$-invariant differential operator 
${\mathbb D}^J={\mathbb D}^J_{k,m}$
growth-killing for index $m>0$, if for all $n,n', r,r'
\in {\mf R}$ with 
$\left( \begin{smallmatrix} n & \frac{r}{2}\\
\frac{r}{2} & m\end{smallmatrix}\right)$ 
and 
$\left(\begin{smallmatrix} n' & \frac{r'}{2}\\
\frac{r'}{2} & m\end{smallmatrix}\right)$ both positive-semidefinite we have 
\begin{equation} \label{DJ}
\mathbb{D}^J \left( v^ke^{-4m\frac{y}{v^2}}\cdot
e^{ 2\pi i((n-n')u +(r-r')x)) } \cdot e^{-2\pi ((n+n')v+(r+r')y}
\right)=0 
\end{equation}
unless $\left(\begin{smallmatrix} n+n' & \frac{r+r'}{2}\\
\frac{r+r'}{2} & 2m\end{smallmatrix}\right) > 0$.
The notion ``growth-killing'' will get justified later on. 

Concerning the case of symplectic groups,
such a differential operator is for $n=1$ given by
\begin{equation} \label{n=1}
 {\mathbb D} = \mbb D_k= 4y^2\frac{\partial}{\partial z}\frac{\partial}{\partial \bar{z}}
-k(k-1)= y^2\cdot\left(\frac{\partial^2}{\partial^2x}+
\frac{\partial^2}{\partial^2y}\right)-k(k-1) 
\end{equation}
for any $k \in \mf R$. For $\spnr$ in general in  \cite{BC}, following \cite{DK} a 
more abstract construction of such 
operators was described. Even earlier, for the special case of degree $2$, 
Maa{\ss} 
\cite{Ma} gave an explicit
formula for such an operator (with $\mrm{Id}$ denoting the identity map)
\begin{equation} \label{M}
 {\mathcal M} = 
\Phi_2+(k-\frac{1}{2})(k-\frac{3}{2})\left(\Phi_1+k(k-2)\cdot \mrm{Id}\right) 
\end{equation} 
where $\Phi_1, \Phi_2$ are $\spnr$-invariant differential operators defined by 
\begin{align} \label{phi12}
\Phi_1:= \tr\left( (Z-\bar{Z})\cdot \{(Z-\bar{Z})\partial \bar{Z}\}^t \cdot 
\partial Z\right)
\text{   and   }\ 
\Phi_2= \det(Z-\bar{Z})^{\frac{5}{2}}\partial^{[2]}\overline{\partial}^{[2]}
 \det(Z-\bar{Z})^{-\frac{1}{2}}.
\end{align} 
with
$$\partial Z:=\left(\begin{array}{cc} \frac{\partial}{\partial z_1}&
\frac{1}{2} \frac{\partial}{\partial z_2}\\
\frac{1}{2} \frac{\partial}{\partial z_2} &
\frac{\partial}{\partial z_4}\end{array}\right)$$
and $\partial^{[2]}:= \det(\partial Z)$.

The main reason to consider such growth killing differential 
operators is that
they allow us to implement in a convenient way the Rankin-Selberg method for non-cusp forms and also to be able to define an extended Petersson
inner product for arbitrary modular forms (not only for cusp forms!),
e.g. for $n=1$ we may define, for $f,g\in M_k(\sltwo)$
$$\{f,g\}:= -\int_{\sltwo\backslash \mf H_1} {\mathbb D_k}(f\cdot \bar{g} 
y^k) \frac{dxdy}{y^2}= 
 -\int_{\sltwo\backslash \mf H_1} {\mathcal K}(f,g) 
 \frac{dxdy}{y^2}$$ 
where
\begin{equation}
{\mathcal K}(f,g)= 4 f'\cdot \bar{g'}y^{k+2} +
2ik(f'\cdot\bar{g}-f\cdot \bar{g'})y^{k+1}.\label{residual}
\end{equation}
The integrand is then exponentially decaying for $y\to \infty$, in particular,
the integral converges, and it generalizes the usual Petersson inner product
$\langle \, ,  \rangle$ for cusp forms (up to a constant), i.e. $\{f,g\}=k(k-1)\cdot  \langle f,g \rangle$
if $f,g$ are cuspidal.

Note that it is unclear in general, whether this hermitian 
form is nondegenerate. In any case, we cannot expect it to be positive definite.
For a discussion of such matters, mostly in the more general context
of Siegel modular forms, see \cite{BC,Za}.

As far as we know, such growth-killing operators were not yet considered
for Jacobi groups. For some purposes, one can use the theta expansion
of Jacobi forms to reduce the problem to ordinary modular forms.
For us, such an approach is not sufficient,
we need a relation between such growth killing operators for
$\sptwor$ and $G^J(\mf R)$.

\subsection{Growth-killing operators on $\mf H_1 \times \mf C$}
We start from an arbitrary smooth 
function $f$ on ${\mf H_1}\times {\mf C}$, later on it will
be of type
$f(\tau,z)= \phi(\tau,z)\cdot \overline{\psi(\tau,z)}$ with Jacobi forms
$\phi,\psi$ of weight $k$ and index $N$.
Then we associate to $f$ two functions on  
${\mf H_1}\times {\mf C}$ and ${\mf H}_2$ as follows:
\begin{equation}
h(\tau,z):= f(\tau,z)\cdot v^k\cdot e^{-4\pi N \frac{y^2}{v}}
\label{h}
\end{equation}
\begin{equation}H(Z):= f(\tau,z)\cdot \det(Y)^k\cdot e^{-4\pi Nv'}
\label{H}
\end{equation}
It is the convenient to describe the connection between them  by the coordinate
\[ t:= v'-\frac{y^2}{v}, \]
which is invariant under the action of $G^{J}({\mf R})$. In fact $H(Z)= h(\tau,z)\cdot t^k\cdot e^{-4\pi Nt}$ so that
\[  H(\widetilde{M} \langle Z \rangle)= h(M \langle (\tau,z \rangle)\cdot t^k e^{-4\pi Nt} \]
holds for all $M\in G^J({\mf R})$ when embededed as $\widetilde{M}$ (cf. \eqref{embed}) in
$\sptwor$.

We consider now ${\mathcal M}(H)$: 
By inspection, taking into account the form of the differential operators
$\Phi_1$ and $\Phi_2$, in particular their degrees in $\partial_4$ ($=\frac{\partial}{\partial z_4}$) and 
$\overline{\partial_4}$ ($=\frac{\partial}{\partial \bar{z_4}}$), we get that we can write ${\mathcal M}(H)$ in the form
\[ {\mathcal M}(f\cdot e^{-4\pi N\cdot v'} \cdot \det(Y)^k)= \]
\[  \{D_0^J
(f\cdot v^k e^{-4\pi N \frac{y^2}{v}}) t^k
+ D_1^J
(f\cdot v^k e^{-4\pi N \frac{y^2}{v}}) t^{k+1}
+D_2^J(f\cdot v^k e^{-4\pi N \frac{y^2}{v}}) t^{k+2}\}\cdot
e^{-4\pi Nt}  \]
with $t$ as above.
Here the $D^J_r$ are automatically invariant differential   
operators for the Jacobi group
(w.r.t $k$ and $N$) with growth killing property.

It is somewhat painful to determine these Jacobi differential operators
explicitly in terms of generators of the ring of Jacobi differential 
operators.
We carry that computation out in the appendix and just quote the result
in a weak form, sufficient for our purpose:
\begin{eqnarray}
D^J_0 &=&0\label{DJ0}\\
D^J_1 & = & 2R(k-\frac{1}{2})(L_1+r)+ R^J_1 \label{DJ1}\\
D^J_2 & = & R^2\cdot (L_1+r)+R^J_2 \label{DJ2}
\end{eqnarray}

Here 
\[ r:=-(k-\frac{1}{2})(k-\frac{3}{2}), \q R:=-4N\cdot \pi.\]
and $R^J_1,R^J_2$ denote some invariant Jacobi differential operators which turn out
to be inessential for us, see subsection~\ref{fudge}.
The essential Jacobi differential operator $L_1$
is explictly given (following \cite{Ochiai}) by
\begin{eqnarray*}
L_1&:=&-(z_1-\bar{z_1})^2
{\partial_1\bar{\partial_1}
}
- (z_2-\bar{z_2})^2\di\dib \\
&&
-(z_1-\bar{z_1})(z_2-\bar{z_2})(\bar{\partial_1}\di+\partial_1\dib),
\end{eqnarray*}  
where we put $\partial_1:=\frac{\partial}{\partial z_1}$
and similarly for $\bar{\partial_1}$.
\subsubsection{A first vanishing result} \label{fudge}
The ``inessential'' growth killing Jacobi differential operators
from above share the property that
all their summands (when written as polynomials in derivatives with possibly
nonconstant coefficients on the left) involve 
 
$$\left(\frac{\partial}{\partial x}\right)^{\alpha} \quad \mbox{or}\quad
\left(\frac{\partial}{\partial y}\right)^{\beta} \quad \mbox{or}\quad
y\frac{\partial^2}{\partial^2 x} \quad \mbox{or}\quad
y\frac{\partial^2}{\partial^2y}$$ 
with nontrivial $\alpha$ or $\beta$, see appendix.
Let us denote this property by (*).

\begin{prop}
Let ${\mathcal D}$ be a 
growth killing Jacobi differential 
operator w.r.t. $k$ and $m$ with property (*). 
Then for any Jacobi forms $\phi,\psi\in J_{k,m}$
we have
\begin{equation}
\int_{\Gamma^J\backslash {\mf H}_1\times {\mf C}}
{\mathcal D}(\phi\overline{\psi}e^{-4\pi N\frac{y^2}{v}}v^k) 
\frac{dudxdvdy}{v^3}=0
\label{vanish}
\end{equation}  
\end{prop}

\begin{proof} We follow the strategy of Eichler-Zagier \cite{EZ}, who 
showed how to
relate a Petersson product of two Jacobi cusp forms to the
Petersson products of the half-integral weight cusps forms
associated with them: 

The  fundamental domain for $\Gamma^J$ can be described as follows:
\begin{itemize} 
\item
$\tau = u+ i v\in \sltwo\backslash {\mf H_1}$
\item $z=x+iy$ with $x \bmod 1$, $y\bmod v$
\end{itemize}
It is the integration over $x$ and $y$ (over compact domains) , which is responsible for the vanishing:\\
We employ the theta expansions
of $\phi_N,\psi_N$:
\begin{equation} \label{theta}
\phi_m(\tau,z)=\sum_{\mu} g_{\mu}\vartheta_{m,\mu}, \q \psi_m(\tau,z)=\sum_{\mu} h_{\mu}\vartheta_{m,\mu}
\end{equation}
where the theta functions are defined as in section~\ref{j-prel}.

Integration over $x\bmod 1$ yields for arbitrary nonnegative integers $\alpha, l,l'$ that 
\[ \int_0^1 \frac{\partial^{\alpha}}{\partial^{\alpha} x}
\left(e^{2\pi i((\nu+2ml) z -\overline{(\mu+2l'm)} \bar{z})} \right)dx = \delta_{0,\alpha}
\delta_{\nu,\mu} \delta_{l,l'}e^{-4\pi (\mu+2ml) y}, \]

therefore the summands involving 
$\frac{\partial}{\partial x}$ or
$y\frac{\partial^2}{\partial^2 x}$
do not contribute.

We fix $\mu=\nu$ and we expand ${\mathcal D}$ as
\[  {\mathcal D}={\mathcal D}_0+\cdots, \]
where ${\mathcal D}_0$ is the part of ${\mathcal D}$ free of derivatives w.r.t.
$x$. We point out that ${\mathcal D}_0$ is still an invariant differential
operator for the substitution
$(\tau,z)\mapsto (\tau,z+\lambda \tau)$, when applied to a function 
which does not depend on $x$.

We are left with
\[ \int_{\sltwo\backslash {\mf H}_1}\int_{{\mf Z}v\backslash {\mf R}} 
{\mathcal D}_0\left(g_r\overline{h_r} 
\sum_{l\in {\mf Z}}e^{-4\pi\frac{m}{v}(y+\frac{r}{2m}+l)^2}
v^k \right) dy \frac{dudv}{v^3}. \]

The summation over $l\in {\mf Z}$ and the integral over $y$
become
\begin{equation}\int_{-\infty}^{\infty} \Delta( e^{-4\pi m\frac{y^2}{v}})dy=
\left\{\begin{array}{ccl}
\sqrt{\frac{v}{4m}}&\mbox{if}& \Delta= id\\
0&\mbox{if}& \Delta= \frac{\partial}{\partial y}
\,\mbox{or}\, \Delta=y\frac{\partial^2}{\partial y^2} 
\end{array}\right. \label{zero}
\end{equation}
The conclusion concerning \eqref{vanish} follows.
Note that we used only the integrations w.r.t. $x$ and $y$,
which go over compact domains. The growth killing property was
not really used. 
\end{proof}

The first part of (\ref{zero}) will be used later on. We are now able to define an extended Petersson inner product on Jacobi forms, using the growth killing differential operator
\begin{equation} \label{DJl1}
 \mc D^J:= L_1+r .
 \end{equation}
\begin{defi} \label{Pdef}
For any Jacobi forms $\phi,\psi$ of weight $k$ 
and index $N$ we put
\[\{\phi,\psi\}:= \int_{\Gamma^J\backslash {\mf H_1}\times {\mf C}}
\mc D^J(\phi\bar{\psi} v^ke^{-4\pi N\frac{y^2}{v}}) \frac{dudxdy}{y^3}. \]
\end{defi}

\subsection{From Jacobi forms to modular forms of half-integral weight} 
In the previous section we have already done some steps towards
expressing the extended Petersson product of two Jacobi forms
by an integral involving the corresponding (vector-valued) 
modular forms of half-integral weight; we complete this procedure by showing the following.
\begin{prop} \label{j2h}
Let $\mc D^J$ be the growth killing Jacobi differential operator from \eqref{DJl1}. Then 
for all Jacobi forms $\phi,\psi\in J_{k,m}$ with
theta expansions as in \eqref{theta}, we have
\[ \int_{\Gamma^J\backslash {\mf H}_1\times {\mf C}}
\mc D^J(\phi\overline{\psi}e^{-4\pi N\frac{y^2}{v}}v^k) 
\frac{dudxdvdy}{v^3}= \]
\begin{equation} \frac{1}{\sqrt{4m}}\cdot
\int_{\sltwo\backslash {\mf H_1}} 
\sum_{\mu \bmod 2m} 
{\mbb D}_{k-\frac{1}{2}} \left(g_{\mu}\overline{h_{\mu}}v^{k-\frac{1}{2}}\right) 
\frac{dudv}{v^2}
\label{vanish2}
\end{equation}  
\end{prop}
Here $\mbb D_{k- 1/2}$ denotes the growth killing differential operator defined in \eqref{n=1}. This proposition is analogous to \cite[Thm.~5.3]{EZ}: we
reduce the extended Petersson product for Jacobi forms (defined by means
of $\mc D^J$) to the extended Petersson product for modular forms of 
half-integral weight. 

\begin{cor} \label{cusp-eis}
The extended Petersson product $\{ \, ,\}$ for Jacobi forms
of weight $k$ and index $m$ satisfies
\[ \{\phi,\psi\}= d_k\cdot \langle \phi,\psi \rangle \]
for all $\phi, \psi \in J_{k,m}$ such that at least one of them is a cusp form. Here $d_k$ is a  constant
independent of $m$ and non zero unless $k=\frac{1}{2}$ or $k=\frac{3}{2}$.
\end{cor}
The corollary follows from the proposition above and the 
corresponding statement for the growth killing differential operator for 
weight $k-\frac{1}{2}$, see eg. \cite{chiera}. 
A more intrinsic proof of the corollary within the 
theory of Jacobi forms should also be 
possible.

\begin{proof}
We write $\mc D^J=L_1+r$ as
 \begin{equation}
4v^2\frac{\partial^2}{\partial z_1 \partial \bar{z_1}}
+4y^2\frac{\partial}{\partial z_2}\frac{\partial}{\partial\bar{z_2}}
+4vy \left(\bar{\partial_1}\frac{\partial}{\partial z_2}+\partial_1
\frac{\partial}{\partial\bar{z_2}}\right)+r.
\label{summands}
\end{equation}

The integration over $x\bmod 1$ and $y\bmod v$ can be handled
in essentially the same way as in the previous section.
Before finally integrating over $\sltwo\backslash {\mf H_1}$
we should therefore consider the integral

$$\int_{-\infty}^{\infty} (L_1+r)(\sum_{\mu} g_{\mu}\overline{h_{\mu}}
e^{-4\pi m\frac{y^2}{v}}v^k ). dy$$

We do this for fixed $\mu$ and for each of the summands of $L_1+r$
in (\ref{summands}) separately: 
\subsubsection{First summand}
The differential operator ${\partial_1\bar{\partial_1}}$
does not involve $y$, therefore integration over $y$ gives just a
factor $\sqrt{\frac{y}{4m}}$ and we get $\frac{1}{\sqrt{4m}} \cdot 4v^2\partial_1\bar{\partial_1}\left(g\cdot \bar{h} v^{k+\frac{1}{2}}\right)$ equals
\[
\frac{1}{\sqrt{4m}}  \left( 4(\partial_1 g)\cdot \overline{\partial_1h} v^{k+\frac{5}{2}}
+ 2i(k+\frac{1}{2})\left((\partial g)\bar{h}-g\cdot \overline{\partial_1 h}
\right) v^{k+\frac{3}{2}}
+(k+\frac{1}{2})(k-\frac{1}{2})g\cdot\bar{h}\cdot v^{k+\frac{1}{2}}   \right). \]

\subsubsection{Second summand}
First we mention that 
the identity
\[ \left(y^2\frac{\partial^2}{\partial y^2}\right)(F)=
\frac{\partial}{\partial y}\left(y^2\cdot 
\frac{\partial}{\partial y}F\right)
+2F-2 \frac{\partial}{\partial y}(y\cdot F) \]
holds (for any reasonable function $F$, say smooth, on ${\mf R})$ .
Therefore - using (\ref{zero}) - 
the contribution of the second summand to the integral over $y$
is (again up to the factor $\frac{1}{\sqrt{4m}}$)
\[ 2\cdot g\cdot \bar{h} v^{k+\frac{1}{2}} . \]

\subsubsection{Third summand}
For $F$ as above, we use $y\frac{\partial}{\partial y}(F)=
\frac{\partial}{\partial y}\left(y\cdot F\right) - F$.

Therefore, the contribution of first part of third summand becomes ${2i}v
\bar{\partial_1}\left(g\cdot\bar{h}v^{k+\frac{1}{2}}\right)$ which simplifies to the expression
\[ {2i}g\cdot (\overline{\partial_1 h})v^{k+\frac{3}{2}}-
(k+\frac{1}{2})g\cdot\bar{h}v^{k+\frac{1}{2}} .\]

In a similar way, we get for the second part the contribution
\[ -{2i}(\partial_1g)\cdot (\overline{ h})v^{k+\frac{3}{2}}-
(k+\frac{1}{2})g\cdot\bar{h}v^{k+\frac{1}{2}}. \]

\subsubsection{Fourth summand}
Obviously, the contribution here is
$$\frac{r}{\sqrt{4m}} g\overline{h}v^{k+\frac{1}{2}} .$$

We collect all contributions to get the integrand for
the modified Petersson product 
(note that the change from $\frac{dudxdydv}{v^3}$
to $\frac{dudv}{v^2}$ changes the power of $v$)

\[ \frac{1}{\sqrt{4m}}\left(4(\partial_1 g)\cdot \overline{\partial_1h} v^{k+\frac{3}{2}}
+ 2i(k-\frac{1}{2})\left((\partial_1 g)\bar{h}
-g\cdot \overline{\partial_1 h}
\right) v^{k+\frac{1}{2}}\right) \]
which is exactly--up to the factor $\frac{1}{\sqrt{4m}}$ -the growth killing operator for functions of weight 
$k-\frac{1}{2}$. 
Note that the contributions involving $g\cdot\bar{h} v^{k-\frac{1}{2}}$
get cancelled.
\end{proof}

\subsection{Lemmas on growth killing operators: boundedness properties} 
Let $\mc M$ denote the invariant differential operator $\mf{H}_2$ from \eqref{M} studied by Maa{\ss} \cite{Ma}  and $F,G \in M^2_k$. We would prove that $\mc M$ applied to certain Fourier series produces a decay, see the lemma below. This justifies our definition of `growth-killing' from section~\ref{gr-defi}. Even though the results in this section might be known to be true intutively, but seems not to be written down. Anyway we take some care to provide complete proofs, since they are crucial for our further investigations. The following proof easily generalises to any degree $n$.

\begin{lem} \label{phib}
For some $c>0$ and all $Z \in \mf H_2$,
\begin{align}
\mc M\left( F(Z) \overline{G(Z)} \det(Y)^k \right) \ll_k \exp (-c \, \tr Y) .
\end{align}
\end{lem}

\begin{proof}
Since $\mc M$ is $\sptwo$-invariant, its enough to consider $Z$ in $\mathcal{F}_2$, the standard Siegel fundamental domain of degree $2$. Let the differential operator $\mc R_2$ be defined by the relation $\det(Y)^k \mc R_2 F := \mc M (\det(Y)^k F)$.
We start from the Fourier expansion of $\Phi_2\left( F(Z) \overline{G(Z)} \det(Y)^k \right) $:
\begin{align} \label{phif}
\underset{ \substack{ S,T \geq 0\\ S+T>0 }  }\sum a_F(S) \overline{a_G(T)} \det(Y)^k \mathcal{R}_2 ( e^{2 \pi i (S-T)X } e^{- 2 \pi \tr (S+T)Y} ).
\end{align}

Let us put in \eqref{phif} $R = S+T$ and $L = S-T$, so that $R>0$ and $R \geq L$. Since the Fourier coefficients are supported on $\Lambda_2 $, the expression \eqref{phif} can be written as
\begin{align} \label{phirl}
&\underset{  R>0, R \geq \pm L  }\sum a_F \left( \frac{R+L}{2} \right) \overline{a_G \left( \frac{R-L}{2} \right)} \det(Y)^k \mathcal{R}_2 ( e^{2 \pi i LX } e^{- 2 \pi \tr RY} ) 
\end{align}
Now $Y$ is Minkowski-reduced, so that we have, for some $\delta_2>0$, the inequalities
\begin{align}
 \delta_2^{-1} Y \leq \mathrm{diag}(y_1,y_2) \leq \delta_2 Y; \q 
 y_\nu \geq \sqrt{3}/2 \q \text{ for  } \nu = 1, 2,
 \end{align}
where in the rest of the proof we put $M_i := M_{i,i}$ for a matrix $M$.
This implies that for some $\delta>0$
\begin{align}
\tr(RY) \geq \delta \, \tr(R); \q \tr(RY)  \geq \delta \, \tr(Y),
\end{align}
where in the second inequality, we take into account that $R$ is half integral.

Since the Fourier series for $F$ and $G$ converge absolutely for $Z=\frac{\delta}{4} i 1_2$, we may trivially estimate the Fourier coefficients by
\begin{align} \label{trest}
a_F(S) \ll \exp( \pi \delta/2 \, \tr(S)); \q a_G(T) \ll \exp( \pi \delta/2 \, \tr(T)).
\end{align}

Let $ \breve{L}$ denote the symmetric matrix obtained by taking absolute values in the entries of $L$. Using \eqref{trest} in \eqref{phirl} and bounding absolutely, we get that $\Phi_2\left( F(Z) \overline{G(Z)} \det(Y)^k \right) $ is at most (note that $X$ is bounded)
\begin{align}
\underset{R>0, R \geq \pm {L}}\sum \mc{P}(Y) \mc{Q}(R, \breve{L}) e^{- 2 \pi \tr(RY)}  ;
\end{align}
where $\mc{P}, \mc{Q}$ are polynomials in their respective arguments.
Now $\mc{P}(Y) \ll \tr(Y)^\ell$, $\mc{Q}(R, \breve{L}) \ll \tr(R)^{\ell'}$. To see this, note that $|L_i| < R_i$ follows from $ \mp L \leq R$, and moreover when $i \neq j$, one has the inequalities
\begin{align} \label{rij}
 2(R_i+R_j) \geq R_i\pm L_i + R_j \pm L_j \geq 2 |R_{i,j} \pm L_{i,j}| \geq 2 (|L_{i,j}| - |R_{i,j}|), 
\end{align} 
which imply that $|L_{i,j}| \leq \frac{3}{2} (R_i+R_j)$. So finally we are reduced to bounding
\begin{align}
\underset{R>0}\sum \left( \# \{L=L^t \mid \pm L \leq R\} \right) \tr(Y)^{\ell'} \tr(R)^\ell e^{- \pi \delta/2 \tr(R)}  e^{- \pi \delta \tr(Y) }  \label{est}
\end{align}
From \eqref{rij}, one has the following inequality: 
\[  \# \{L=L^t \in \symtwor \mid \pm L \leq R\} \ll \det(R)^{\frac{3}{2} }. \]

To see this, note that we can suppose that $R$ is reduced and then can easily deduce the number of possibilities for $L$ are atmost
\[ \prod_i (2 r_i+1) \prod_{i<j} (2 (r_i r_j)^{1/2} +1) \ll \det(R)^{\frac{3}{2} }. \]

Thus for some $\beta>0$, we would have to bound $\underset{R>0}\sum \tr(R)^{\beta } e^{- \pi \delta \tr(R)} e^{- 2 \pi \delta \tr(Y)} $. Putting $\tr(R) = r$ in the above and noting that $\# \{ R \in \Lambda^+_n \mid \tr(R)=r  \} \ll r^{3}$, we have for some $\beta'>0$, the bound
\[
\sum_{r=1}^\infty r^{\beta'} e^{- \pi \delta r} \cdot e^{- 2 \pi \delta \tr(Y)} \ll e^{- 2 \pi \delta \tr(Y)} .   \qedhere \]
\end{proof}

\begin{lem} \label{phij-lem}
\begin{align} \label{phij}
\mc M \left(( \phi_m(\tau,z) \overline{\psi_m(\tau,z)} e^{- 4 \pi m v'} \det(Y)^k \right) \ll 1.
\end{align}
for all $ \left( \begin{smallmatrix}\tau & z \\ z & \tau' \end{smallmatrix}\right) \in \mathcal{C}_{2,1} \backslash \mf H_2$, where the implied constant depends only on $k$.
\end{lem}

\begin{proof}
This follows from the proof of \propref{phib}. First we observe that the quantity in \eqref{phij} is
\[  \underset{S \geq 0  ,T \geq 0}\sum a_F\left( S\right)  \overline{a_G\left(  T\right) } \Phi_2 \left(  e^{2 \pi i (S-T)X} e^{- 2 \pi  (S+T)Y} \right), \]
where the right lower entries of $S$ and $T$ is $m$. This is a subseries of what is considered in \propref{phib}, and hence the proposition follows since we bounded the latter series absolutely.
\end{proof}

\section{The Eisenstein part of the Fourier-Jacobi coefficients of $E_{2,1}(f, \cdot)$} \label{Kleisen}
So far, we only dealt with Jacobi forms.
In subsequent sections we aim at properties of the modified Petersson product
for Jacobi forms
$\{\phi,\psi\}$
when $\phi=\phi_N$ and $\psi=\psi_N$ arise as Fourier-Jacobi coefficients
of Siegel modular forms $F$ and $G$ of degree $2$ and weight $k$. We start with the case of Klingen Eisenstein series $E_{2,1}(f, \cdot)$, abbreviated as $[f]$ for convenience.
 
The goal of this section is to express the Eisenstein part $\mathcal{E}_{k,m}$ of the $m$-th Fourier Jacobi coefficient $\phi_m$ of the Klingen Eisenstein series $[f]$ attached to a cuspform$f \in S_k$ in terms of Jacobi Eisenstein series at the cusp `$\infty$'; and eventually in the form $E_{k,1} \vert_{k,1} \mc T_m$ for some explicitly defined Hecke operator $\mc T_m$, where $E_{k,1} \in J_{k,1}$ denotes the Eisenstein series. 
We did not expect such a result beforehand, but it is crucial for our work.
These formulas would be used to compute and estimate the norm of $\mc E_{k,m}$ with respect to the extended inner product.

Throughout this section, let us write, following the notation in \cite{EZ}, $m=ab^2$ with $a$ square--free. Further let $E_{k,m,s}, \ (0 \leq s \leq b/2)$ be the Eisenstein series attached to the cusp parametrised by $s$, so that $E_{k,m}=E_{k,m,0}$. We now recall the formula for the `degenerate' Fourier coefficients (i.e., those $(n,r)$ for which $4nm=r^2$) for these Eisenstein series. Let 
\begin{align}
E_{k,m,s}(\tau,z) = \underset{ \substack {n \geq 0, r \in \mf Z \\ 4mn \geq r^2} }\sum  e_{k,m,s}(n,r) e(n \tau+rz).
\end{align}
Let us recall from \cite{EZ} that
\begin{prop} \label{ekms}
Let $n,r$ be such that $4mn=r^2$. Then one has $e_{k,m,s}(n,r) = \epsilon_{m}(r)$, where
\begin{align} \label{ekmseq}
\epsilon_{m,s}(r) = \begin{cases}  \tfrac{1}{2} \q \text{if} \ r \equiv \pm 2abs \pmod{2m} \text{ and } 2s \not \equiv 0 \bmod b, \\ 1 \q \text{if} \ r \equiv \pm 2abs \pmod{2m} \text{ and } 2s  \equiv 0 \bmod b,
\\ 0 \q \text{otherwise} . \end{cases}.
\end{align}
\end{prop}

\begin{prop} \label{ekm}
\begin{align}
\mathcal{E}_{k,m} = \underset{1 \leq s \leq b}\sum a_f(a \cdot (s,b)^2) E_{k,m,s}.
\end{align}
\end{prop}

\begin{proof}
It is known from \cite{EZ} that when $k$ is even, the set $\{ E_{k,m,s} \}_{0 \leq s \leq b/2}$ is a basis for $J_{k,m}^{\mrm{Eis}}$ (see \cite{EZ}). Thus for scalars $q_s$, we can write
\begin{align}
\mc E:= \mathcal{E}_{k,m} = \underset{0 \leq s \leq b/2}\sum q_s E_{k,m,s} .
\end{align}

Since $k$ is even, it is clear from the description of the `degenerate' Fourier coefficients of $E_{k,m,s}$ that two such coefficients, say, $e_{k,m,s}(n,r)$ and $e_{k,m,s'}(n,r)$ are equal if and only if $s=s'$, for any $0\leq s,s' \leq b/2$. Thus one obtains for any $r \in \mf Z$ such that $r \equiv 2abs \pmod{2m}$:
\begin{align}
  q_s  \cdot \epsilon_{m,s}(r) =  c_{\mc{E}}(\frac{r^2}{4m} , r) .
\end{align}

Since we are only concerned with degenerate Fourier coefficients, its clear that
\begin{align} \label{e=phi}
c_{\mc E_{k,m}} =  c_{\phi_m} (n,r),
\end{align}
where $\phi_m = \mc E_{k,m} + \phi^\circ_m$; with $\phi^\circ_m \in J^{\mrm{cusp}}_{k,m}$.

On the other hand, choosing $r = 2abs$ gives
\begin{align*}
 c_{\phi_m}(\frac{r^2}{4m} , r) =  a_{[f]} \left( \begin{pmatrix}
as^2 & abs \\ abs & ab^2
\end{pmatrix} \right).
\end{align*}

Furthermore, we know that 
\begin{align} \label{Phif}
 f = \Phi ([f]) = \underset{n \geq 1}\sum a_{[f]} \left( \begin{pmatrix}
n & 0 \\ 0 & 0
\end{pmatrix} \right) q^n.
\end{align}

Next we observe that for some $\alpha \geq 1$,
\begin{align} \label{alphas}
\begin{pmatrix}
as^2 & abs \\ abs & ab^2
\end{pmatrix} = U' \begin{pmatrix}
\alpha & 0 \\ 0 & 0
\end{pmatrix}  U, 
\end{align}
for some $U \in \mrm{GL}_2(\mf Z)$. For $T= \begin{psm} n & r/2 \\ r/2 & m \end{psm} \in \Lambda^+_2$, let us define $c(T):= \mrm{gcd}(n,r,m)$. Since the action of $\mrm{GL}_2(\mf Z)$ preserves the content of a half-integral matrix, we get
\begin{align}
\alpha = \mrm{gcd} (as^2,abs, ab^2) = a \cdot (s,b)^2 .
\end{align}

Hence from \eqref{e=phi}, \eqref{Phif} and \eqref{alphas}, we get
\[  q_s \cdot \epsilon_{m}(2 abs) =  a_f(a \cdot (s,b)^2).  \]

Finally noting that $E_{k,m,s}=E_{k,m,-s}$, \eqref{ekmseq} and that $E_{k,m,s}$ depends only on $s \bmod b$, we easily obtain the statement of our proposition.
\end{proof}

We recall the Hecke operator $U_l$ ($l \geq 1$) defined on $J_{k,m}$ by
\[ \phi \mid U_l = \phi(\tau, lz), \]
which maps $J_{k,m}$ to $J_{k, ml^2}$. We next compute the image of the standard Eisenstein series $E_{k,m}$ under the operator $U_\ell$.
\begin{lem} \label{ekmu}
\begin{align}
 E_{k,m} \mid U_\ell = \underset{ \substack { 1 \leq s \leq b\ell  \\ b \mid s  } }\sum E_{k, m \ell^2,s}.
\end{align}
\end{lem}

\begin{proof}
If $e'(n,r)$ denote the Fourier coefficients of $E_{k,m} \mid U_\ell $, then we know (see e.g., \cite{EZ}) that $e'(n,r) =  e_{k,m}(n, r/\ell) $. Thus $e'(n,r)$ equals $1$ if  $2m \mid r/\ell$ and $0$ otherwise, provided that $(n,r)$ is `degenerate' with respect to the index $m \ell^2$.

It is easy to see (e.g. by an unfolding argument or simply by using the remarks in \cite[p.~26]{EZ}) that $U_\ell$ maps $J^{\mrm{Eis}}_{k,m }$ to $ J^{\mrm{Eis}}_{k,m \ell^2}$. Since $k$ is even, let us write
\[ E_{k,m} \mid U_\ell = \underset{0 \leq s \leq b \ell/2}\sum t_s E_{k,m\ell^2,s}, \]
and argue as in the proof of \propref{ekm}, i.e., we compare the $(r^2/4m\ell^2, r)$ coefficients and choose $r=2ab\ell s$. Then $2m \mid r/\ell$is equivalent to $b \mid s$ and we infer that $t_s$ equals $2$ if $b$ divides $s$ and $0$ otherwise. The proof follows, again, by using the facts that $E_{k,m,s}=E_{k,m,-s}$, \eqref{ekmseq} and that $E_{k,m,s}$ depends only on $s \bmod b \ell$.
\end{proof}

Following the notation in \cite{haya}, let us put
\begin{align} \label{gkm}
g_k(m) = \sum_{y^2 \mid m} \mu(y) \sigma_{k-1}(m/y^2).
\end{align}
By multiplicativity, one checks easily that $g_k(m)=m^{k-1} \underset{p \mid m}\prod (1+p^{-k+1}) $. Let us now recall the following relation between the Jacobi Eisensein series $E_{k,m}$ and $E_{k,1}$ (see \cite{EZ}).
\begin{align} \label{ekmek1}
E_{k,m} = g_k(m)^{-1} \underset{t^2 \mid m}\sum \mu(t) E_{k,1} \mid U_{t} \circ V_{\frac{m}{t^2}}.
\end{align}

The operator $V_N$ would be defined and discussed in section~\ref{VN}, we do not need it here. By combining \propref{ekmek1} with the next theorem, we are going to give a formula for $\mathcal{E}_{k,m}$, which upon using \eqref{ekmek1}, ultimately shows that these objects are determined by $E_{k,1}$ acted upon by certain Hecke operators $\mc T_m$, i.e., $\mc E_{k,m} = E_{k,1} \vert_{k,m} \mc T_m$. This would be useful later.

\begin{thm} \label{ekmfor} Keeping the notation introduced in this section, one has
\begin{align} \label{ekmfor1}
\mathcal{E}_{k,m} =  \underset{ \substack{ 1 \leq \lambda \leq b \\ \lambda \mid b }   }\sum a_f( a \lambda^2 ) \underset{ \substack{1 \leq d \leq b \\ d \lambda \mid b  } }\sum \mu(d) E_{k, a\lambda^2 d^2} \mid U_{\frac{b}{\lambda d}}.
\end{align}
\end{thm}

\begin{proof}
We start from the expression of $\mathcal{E}_{k,m} $ from \propref{ekm}:
\begin{align} \label{ekm1}
\mathcal{E}_{k,m} = \underset{1 \leq s \leq b}\sum a_f(a \cdot (s,b)^2) E_{k,m,s}.
\end{align}
Let $\lambda = (s,b)$ and write $s = \lambda s', b = \lambda b'$. We can then rewrite \eqref{ekm1} in the following way, by summing over fixed values of $\lambda$:
\begin{align} \label{ekm2}
\mathcal{E}_{k,m} &= \underset{  \lambda \mid b   }\sum     a_f(a \lambda^2) \underset{  (s,b)=\lambda }\sum  E_{k,m,s}\\
&  = \underset{   \lambda \mid b    }\sum     a_f(a \lambda^2) \underset{  (s',b')=1 }\sum  E_{k,m, \lambda s'}.
\end{align}

Removing the coprimality condition by using the $\mu(\cdot)$ function, the second summand in the above can be rewritten as
\begin{align}  
& \underset{ 1 \leq s' \leq b'}\sum \q \underset{d \mid (s',b')}\sum \mu(d) E_{k,m,\lambda s'} \\
&=  \underset{  d \mid b'  }\sum  \mu(d)   \underset{ \substack{ 1 \leq s' \leq b' \\ d \mid s'} }\sum E_{k,m,\lambda s'}\\
& = \underset{  d \mid b'  }\sum  \mu(d)   \underset{ 1 \leq s'' \leq b''  }\sum E_{k,m,\lambda d s''},
\end{align}
where we have put $s' = d s'', b' = d b''$.

By \lemref{ekmu}, we see that the second sum in the above is just
\[  E_{k,a \lambda^2 d^2} \vert U_{b''}, \]
and this immediately gives the theorem.
\end{proof}

\subsubsection{Expression in terms of Fourier Jacobi coefficients of Siegel Eisenstein series}
Recall the function $g_k(m)$ from \eqref{gkm}. For $f \in M^1_k$, define the functions
\begin{align} \label{alphadef}
g_f(m)= \underset{d^2 \mid m}\sum \mu(d) a_f(m/d^2); \q   \alpha_m(t;f) = \underset{\ell \mid t}\sum \mu(\frac{t}{l}) \frac{g_f(m/\ell^2) }{g_k( m/\ell^2) }.
\end{align}
Let $e_{k,m}$ denote the Fourier Jacobi coefficients of the Siegel Eisenstein series $E^k_{2,0}(Z)$, so that
\begin{align} \label{siegel-fj}
E^k_{2,0}(Z) = \ut{m \geq 0}\sum  e_{k,m}(\tau,z) e(m \tau').
\end{align}

We can now give an expression for $\mc E_{k,m}$ in terms of $e_{k,m}$, which is useful in certain circumstances, e.g., while dealing with asymptotics of Fourier coefficients.
\begin{lem} \label{E2e}
$\mc E_{k,m} = \underset{t^2 \mid m}\sum  g_f(m/t^2) E_{k,m/t^2} \vert U_t.$
\end{lem}

\begin{proof}
We start from \eqref{ekmfor1} and make a change of variable $\lambda d=y$ and interchange the order of summations therein to obtain (recalling that $m=ab^2$)
\[  \mc E_{k,m} = \ut{x \mid b}\sum \, \, \ut{d \mid b/x} \sum \mu(d) a_f( \frac{m}{d^2 x^2} ) E_{k,m /x^2} \vert U_x.\]
Since $x \mid b$ is equivalent to $x^2 \mid m$, we immediately get the lemma.
\end{proof}

\begin{prop} \label{e2E}
$\mc E_{k,m} =  c_k^{-1} \underset{t^2 \mid m}\sum \alpha_m(t) e_{k,m/t^2} \vert U_t$.
\end{prop}

\begin{proof}
We first recall a result 
from \cite{BFJ} about the Fourier Jacobi coefficients of the Siegel Eisenstein series of degree $2$, stated in a simple form as in a paper by Hayashida \cite{haya}:
\[  e_{k,m} = c_k \ut{d^2 \mid m}\sum  g_k(m/d^2) E_{k,m/d^2} \vert U_d,\]
where $c_k= 2/\zeta(1-k)$.

Let us invert the above relation to express $E_{k,m}$ in terms of $e_{k,m}$; we omit the proof, which can be checked easily by a direct calculation. We get
\[  E_{k,m} = \frac{c^{-1}_k}{g_k(m)} \ut{d^2 \mid m}\sum  \mu(d) e_{k,m/d^2} \vert U_d. \]
The proposition then follows immediately from \lemref{E2e}.
\end{proof}

\section{Interlude on adjoints of some Hecke operators and bounds on eigenvalues} \label{adj-h}

\subsection{Adjoint of $U_l$}
In the subsection we proceed to compute the adjoint $U^*_l$ of the $U_l$ operator with respect to the (extended) Petersson inner product. We provide the details for the convenience of the reader, since these may not be available in the literature explicitly. We start from the definition of $U_l$ to write, for $\phi \in J_{k,m}, \psi \in J_{k,ml^2}$, 
\[  \phi \mid U_l =  \phi \mid_{k,m} \left( \begin{psm}  l^{-1} & 0 \\0 & l^{-1} \end{psm} , [0,0] \right).\]
(Here we use the action of $\mrm{GL}^+_2(\mf Q)$ on functions in the usual way, i.e. with $\alpha= ( \left( \begin{smallmatrix} a & b \\ c &d  \end{smallmatrix} \right), [0,0])$, one puts $\phi \mapsto \phi \mid_{k,m} \alpha = \det(\alpha)^{k/2} (c\tau+d)^{-k} e \left( - m c z^2/(c \tau+d) \right) \phi \left((\alpha(\tau), \frac{z}{c \tau+d} \right)$ and checks that this is a group action.) For $\phi,\psi$ are on a subgroup $G$ of finite index in $\Gamma^J$ of weight $k$ and index $m$, we define the Petersson inner product (see also \cite{KS})
\begin{equation} \label{petersson}
\langle \phi , \psi\rangle := \frac{1}{[\Gamma^J \colon G]} \int_{F_G} \phi(\tau,z) \overline{\psi(\tau,z)} e^{- 4 \pi m y^2/v} v^k dV, 
\end{equation} 
whenever the integral converges, and $dV$ being the invariant volume element on $\mf H_1 \times \mf C$ and $F_G$ denotes a fundamental domain for $G$.

\begin{lem} \label{adjoint}
For cusp forms $\phi, \psi$ on $G$ as above and $\eta = \left( \begin{psm}  l^{-1} & 0 \\0 & l^{-1} \end{psm} , [0,0] \right) \in \mrm{GL}^+_2(\mf Q) \ltimes \mf Q^2 $, one has the formula 
\begin{align}
 \langle  \phi \mid \eta, \psi  \rangle = \langle  \phi , \psi \mid \eta^{-1} \rangle,
\end{align}  
and the same formula holds for any pair of Jacobi forms on $G$ w.r.t. the extended inner product $\{  \,  , \}$.
\end{lem}

\begin{proof}
Let $\phi \in J_{k,m}^{cusp}, \psi \in J_{k,ml^2}^{cusp}$. From the definition \eqref{petersson} it follows, with $F$ a fundamental domain for $\Gamma^J$ that
\begin{align*}
\langle \phi \mid \eta , \psi\rangle &=  \int_{F} \phi(\tau,lz) \overline{\psi(\tau,z)} e^{- 4 \pi m l^2 y^2/v} v^k dV\\
&= l^{-2}  \int_{F'} \phi(\tau,z) \overline{\psi(\tau,z/l)} e^{- 4 \pi m y^2/v} v^k dV\\
&=l^{-2}  \int_{F'} \phi(\tau,z) \overline{\psi(\tau,z) \mid \eta^{-1}} e^{- 4 \pi m y^2/v} v^k dV,
\end{align*} 
where $F'$ can be taken as $\cup_{\tau \in \mathscr F} D'_\tau$ ($\mathscr F = \sltwo \backslash \mf H_1$) such that $D'_\tau := \{ (\tau, p+q\tau \mid p,q \in [0,l] \}$; with the choice of $F$ as $\cup_{\tau \in \mathscr F} D_\tau$ such that $D'_\tau := \{ (\tau, p+q\tau \mid p,q \in [0,1] \} $.

It is readily checked that $\psi \mid \eta^{-1}$ is a Jacobi form on $G:=\sltwo \ltimes l \mf Z^2$ of weight $k$ and index $m$, that a fundamental domain for the group $G$ can be taken to be $F'$; and this completes the proof since $[\Gamma^J \colon  G]=l^2$.

For the second assertion, one notes that the above proof for cusp forms goes through, as our extended Petersson product is also defined by integration over a fundamental domain of $\Gamma^J$ on $\mf H_1 \times \mf C$ of a $\Gamma^J$-invariant function. The difference, if any, is only in notation.
\end{proof}

\begin{rmk}
We point out that the same formula should hold in general, i.e., for all $\eta \in \mrm{GL}^+_2(\mf Q) \ltimes \mf Q^2 $ (also see \cite{KS} for a statement for $\mrm{SL}(2, \mf Q)$; our case does not follow from this).
\end{rmk}

Let us now proceed towards a formula for the operator $U^*_l$ w.r.t. $\{  \, ,\}$. Put $A_l:=\left( \begin{psm}  l & 0 \\0 & l \end{psm} , [0,0] \right)$. Using the lemma we can therefore write
\begin{align}\label{uladj}
\{ \phi \mid U_l, \psi \} =  \{ \phi, \psi \mid \left( \begin{psm}  l & 0 \\0 & l \end{psm} , [0,0] \right) \}.
\end{align}
Since $\psi \mid A_l$ is a Jacobi form on $\sltwo \ltimes l \mf Z^2$ of weight $k$ and index $m$ (cf. proof of Lemma~\ref{adjoint}, with $\eta^{-1}=A_l$), we can now sum both sides of \eqref{uladj} over any set of representatives $(\lambda, \mu)$ of $(\mf Z/ l \mf Z)^2$ to get
\begin{align}
\{ \ ( \sum_{(\lambda,\mu)}  \phi \mid_{k,m} [\lambda,\mu] )\ \mid U_l  , \psi \} &=  \{ \phi, \sum_{(\lambda,\mu)} \psi \mid A_l \mid_{k,m} [-\lambda,-\mu]  \}
\end{align}
so that, 
\begin{align}
l^2 \{ \phi \mid U_l , \psi \} =  \{ \phi , \sum_{(\lambda,\mu)} \psi \mid_{k,m}   \ \left(  \begin{psm}  l & 0 \\0 & l \end{psm}, [\lambda,\mu]  \right) \ \}.
\end{align}

It is then formal to check that $\sum_{(\lambda,\mu) }   \psi \mid_{k,m}  \left(  \begin{psm}  l & 0 \\0 & l \end{psm}, [\lambda,\mu]  \right)$ is in $J_{k,m}$. Thus $U_l^* \colon J_{k,ml^2} \rightarrow J_{k,m}$ is given by
\begin{align}
\psi \mapsto l^{-2} \sum_{(\lambda,\mu) \in (\mf Z/ l \mf Z)^2}  \psi \mid_{k,m} \left(  \begin{psm}  l & 0 \\0 & l \end{psm}, [\lambda,\mu]  \right).
\end{align}
The Fourier expansion of $\psi \mid U_l^*$ can then be computed to be  as follows.

\begin{lem} \label{ul*lem}
\begin{align}
\psi \mid U_l^* (\tau,z) = l^{-1} \underset{(D,r)}\sum \left(  \underset{ \substack{ r' \bmod{2ml} \\ r' \equiv r \bmod{2m} } }\sum c_\psi(l^2D,lr') \right) e\left( \frac{r^2-D}{4m}\tau + rz \right)  .
\end{align} \label{ul*}
\end{lem}

\begin{proof}
From Lemma~\ref{ul*lem}, we see that, for $\psi \in J_{k,ml^2}$, $\psi\mid U_l^*$ equals
\begin{align*}
&l^{-2} \sum_{(\lambda,\mu) \in (\mf Z/ l \mf Z)^2} e^m(\lambda^2 \tau+ 2 \lambda z) \sum_{(D,r)} c_\psi(D,r) e \left( \frac{r^2-D}{4ml^2}  \tau+ (\frac{z +\lambda \tau +\mu}{l} ) r \right) \\
&=l^{-2} \sum_{(D,r)} c_\psi(D,r) \sum_\lambda e \left( ( m \lambda^2 + \frac{r^2-D}{4ml^2} +\frac{r \lambda}{l} ) \tau \right) e \left( (2 \lambda m + \frac{r}{l} )z\right) \sum_\mu e \left(  \frac{\mu r}{l} \right) \\
&= l^{-1} \sum_{\substack{  (D,r) \\ l \mid r} } c_\psi(D,r) \sum_{\lambda \bmod l} e \left(  \frac{(2 \lambda m l +r)^2 -D}{4ml^2} \tau \right) e \left( ( \frac{2 \lambda m l + r} {l} )z\right) .
\end{align*}
After this, making a change of variable $r \mapsto lr$, $D \mapsto Dl^2$ (note that $D \equiv r^2 \bmod{4ml^2})$, gives the desired result. 
\end{proof}

\subsection{Adjoint of $V_N$} \label{VN}
We recall the Hecke operator $V_N$ for $N\geq 1$ in the setting of Jacobi forms of index $m$, which maps $J_{k,1}$ to $J_{k,N}$ and moreover preserves the space of Eisenstein series. In terms of the Fourier expansions it is defined to be the operator which maps
\begin{align*} 
 \underset{ \substack {D<0,r\in \mathbf{Z}  \\ D\equiv r^2 \bmod 4 } }\sum c(D,r) e\left( \frac{r^2-D}{4} \tau + rz \right) 
\mapsto  \underset{ \substack {D<0,r\in \mathbf{Z}  \\ D\equiv r^2 \bmod 4N } }\sum  \big(  \underset{\substack{d \mid (r,N)  \\ D\equiv r^2 \bmod 4Nd }}\sum d^{k-1} c \Big( \frac{D}{d^2}, \frac{r}{d} \Big)  \big)  e\left( \frac{r^2-D}{4N} \tau + rz \right) . \nonumber
\end{align*}

In \cite{KS}, the adjoint $V_N^*$ of $V_N$ w.r.t. $\langle , \rangle$ has been computed, which we recall below. Namely the action of $V_N^*$ as an operator from $J_{k,N}$ to $J_{k,1}$ on the Fourier coefficients is given by
\begin{align} \label{vn*}
&\underset{ \substack {D<0,r\in \mathbf{Z}  \\ D\equiv r^2 \bmod 4N } }\sum c(D,r)e\left( \frac{r^2-D}{4N} \tau + rz \right) \\
& \mapsto \underset{ \substack {D<0,r\in \mathbf{Z}  \\ D\equiv r^2 \bmod 4 } }\sum \Big ( \underset{d \mid N}\sum d^{k-2} \ut{ \substack{s \bmod{2d} \\ s^2 \equiv D \bmod 4d }}\sum c(N^2/d^2,Ns/d) \Big ) e\left( \frac{r^2-D}{4} \tau + rz \right).
\end{align}
An inspection of the proof in \cite{KS} shows that the same formula for $V_N^*$ holds on $J_{k,1}$ w.r.t. the extended product $\{ \, , \}$, as one is able to express $V_N$ by the actions of elements in $\mrm{SL}(2, \mf Q) \ltimes \mf Q^2$ and there is no problem in carrying these over w.r.t. $\{ \, , \}$.

\subsection{Bounds for certain eigenvalues}
For positive integers $\ell_i, q_i$ ($i=1,2$) with $q_i \ell_i^2=m$ and $\bphi = E_{k,1}$, we put
\begin{align}
&\bphi \mid V_{q_1} = \bphi_1, & \bphi \mid V_{q_1} \mid U_{\ell_1} = \bphi_2,\\
&\bphi \mid V_{q_1} \mid U_{\ell_1} \mid U^*_{\ell_2}= \bphi_3, &\bphi  \mid V_{q_1} \mid U_{\ell_1} \mid U^*_{\ell_2} \mid V^*_{q_2} = \bphi_4.
\end{align}

Our next aim is to compute the degenerate Fourier coefficients $c_{\bphi_j}(0,r)$ ($r \in \mf Z$) of these forms by making use of the results of the pervious section. These are as follows. Let $\sigma_{k-1}(n) = \sum_{d \mid n} d^{k-1}$ and $ \mrm{gcd.}(a,b,c)$ be the greatest common divisor of $a,b,c$.
\begin{align}
c_{\bphi_1}(0,r) &= \sigma_{k-1} \big( \mrm{gcd.}( \frac{r^2}{4q_1}, \frac{r}{2}, q_1 )  \big), \nonumber \\
c_{\bphi_2}(0,r) &= \sigma_{k-1} \big(  \mrm{gcd.}( \frac{r^2}{4m}, \frac{r}{2 \ell_1}, q_1 )  \big) , \nonumber \\
c_{\bphi_3}(0,r) &=  \ell_2^{-1} \underset{ \substack{ r' \bmod{2m/\ell_2^2} \\ r' \equiv r \bmod{2m/\ell_2} }}\sum \sigma_{k-1} \big(  \mrm{gcd.}( \frac{r'^2}{4m}, \frac{r'}{2 \ell_1}, q_1 )  \big), \nonumber \\
c_{\bphi_4}(0,r) &=  \ell_2^{-1} \underset{x \mid q_2}\sum x^{k-2}  \underset{ \substack{s \bmod {2x} \\ s^2 \equiv 0 \bmod{4x} } }\sum \ \underset{ \substack{ r' \bmod{q_2} \\ r' \equiv \frac{q_2}s{x} \bmod{2 q_2\ell_2} }}\sum \sigma_{k-1} \big(  \mrm{gcd.}( \frac{r'^2}{4m}, \frac{r'}{2 \ell_1}, q_1 )  \big). \label{phi4}
\end{align}
We note that $c_{\bphi_4}(0,r)$ does not depend on $r$, as it should be, since $\bphi_4 \in J_{k,1}$. Finally, since each of the above operators preserve the space of Eisenstein series and $J_{k,1}$ is one dimensional, we have in the above notation, the following.

\begin{prop} \label{eigenest}
We keep the above notation. Then $\phi_4 = c_{\bphi_4}(0,0) E_{k,1}$; \text{moreover one has }
\[ c_{\bphi_4}(0,0) \ll (q_1q_2)^{k - 5/4}.  \]
\end{prop}

\begin{proof}
We start by bounding the term $\sigma_{k-1}( \ldots)$ by $\sigma_{k-1}(q_1)$ in \eqref{phi4}. This is not really an overkill, because $q_1$ could divide each of the quantities $r'^2/4m,r'/2 \ell_1$. The number of terms in the summation over $r'$ is atmost $2 \ell_2$. Hence
\[ c_{\bphi_4}(0,0) \leq 2 \sigma_{k-1}(q_1) \underset{x \mid q_2}\sum x^{k-2}  \, \# \{ s \bmod {2x} \mid s^2 \equiv 0 \bmod{4x} \} .\]

But it is known that $\# \{ s \bmod {2x} \mid s^2 \equiv 0 \bmod{4x} \} = x_1$ where $x=x_0x_1^2$ with $x_0$ square-free. We recall the short proof. If $N(x)$ denotes the above quantity, then it is known, see eg., \cite{EZ} that
\[ \sum_{x=1}^\infty N(x) x^{-s} = \zeta(s)\zeta(2s-1) /\zeta(2s);\] from which it follows that $N(x) = \underset{f^2 \mid x}\sum f |\mu(x/f^2)| = x_1$, as claimed. This gives us the bound
\[ c_{\bphi_4}(0,0) \ll q_1^{k-1} q_2^{k-3/2}. \]

Using the symmetry of the situation in $q_1,q_2$; i.e., 
\[ \{ \bphi \mid V_{q_1} \mid U_{\ell_1} ,  \bphi \mid V_{q_2} \mid U_{\ell_2}   \}  = \overline{\{ \bphi \mid V_{q_2} \mid U_{\ell_2} ,  \bphi \mid V_{q_1} \mid U_{\ell_1}   \}} \]

we also get
\[  c_{\bphi_4}(0,0) \ll q_2^{k-1} q_1^{k-3/2}; \]

and hence by combining these two, that
\[
 c_{\bphi_4}(0,0) \ll (q_1 q_2)^{k-5/4}. \qedhere \]
\end{proof}

\section{Bounds for extended Petersson norms and implications}
\subsection{Bound for the Petersson norm of $\mc E_{k,m}$} \label{ekm-bd}
In this subsection we would like to estimate the quantity $ \{ \mc E_{k,m} , \mc E_{k,m} \}$. We recall from section~\ref{Kleisen} that
\[ \mathcal{E}_{k,m} = \underset{ \substack{ 1 \leq \lambda \leq b/2 \\ \lambda \mid b }   }\sum a_f( a \lambda^2 ) \underset{ \substack{1 \leq d \leq b/2 \\ d \lambda \mid b  } }\sum \mu(d) E_{k, a\lambda^2 d^2} \mid U_{\frac{b}{\lambda d}}. \]

Hence letting $\lambda_1,\lambda_2,d_1,d_2$ run over obvious ranges, we have
\begin{equation}
\begin{gathered}
 \{ \mc E_{k,m} , \mc E_{k,m} \} \\
 = \underset{\lambda_1,\lambda_2,d_1,d_2}\sum a_f( a \lambda_1^2 ) a_f( a \lambda_2^2 ) \mu(d_1)   \mu(d_2) \{ E_{k, a\lambda_1^2 d_1^2} \mid U_{\frac{b}{\lambda d_1}}, E_{k, a\lambda_2^2 d_2^2} \mid U_{\frac{b}{\lambda d_2}}\}. 
\end{gathered} 
\end{equation}

Bounding absolutely, using Deligne's bound on the Fourier coefficients of elliptic cusp forms, we find that for any $\varepsilon>0$, $\{ \mc E_{k,m} , \mc E_{k,m} \}$ is bounded by
\begin{equation} \label{Ekmest}
 ( a \lambda_1^2)^{\frac{k-1}{2} +\varepsilon} ( a \lambda_2^2)^{\frac{k-1}{2} +\varepsilon} \underset{\lambda_1,\lambda_2,d_1,d_2}\sum \{ E_{k, a\lambda_1^2 d_1^2} \mid U_{\frac{b}{\lambda_1 d_1}}, E_{k, a\lambda_2^2 d_2^2} \mid U_{\frac{b}{\lambda_1 d_2}}\}; 
\end{equation}
where the implied constant may depend on $f$.

To bound the above inner product of the two Jacobi forms on the r.h.s. above, we use  the relation between $E_{k,m}$ and $E_{k,1}$ from \propref{ekmek1} and use the commutativity properties of the operators $U_l$, $V_{l'}$ (see eg. \cite{EZ}) to get
\begin{equation} \label{ekek}
\begin{gathered}
 \{ E_{k, a\lambda_1^2 d_1^2} \mid U_{\frac{b}{\lambda_1 d_1}}, E_{k, a\lambda_2^2 d_2^2} \mid U_{\frac{b}{\lambda_2 d_2}}\}  =\\
   (a \lambda_1^2 d_1^2)^{-k+1} (a \lambda_2^2 d_2^2)^{-k+1} 
 \underset{p \mid a \lambda_1^2 d_1^2}\prod (1+p^{-k+1})^{-1} \underset{p \mid a \lambda_2^2 d_2^2}\prod (1+p^{-k+1})^{-1} \\
\times \underset{ \substack{ t_1^2 \mid a\lambda_1^2 d_1^2 \\t_2^2 \mid a\lambda_2^2 d_2^2 } }\sum \mu(t_1) \mu(t_2) \ \{ E_{k,1} \mid V_{\frac{ a\lambda_2^2 d_2^2 }{t_2^2}} \circ U_{\frac{t_1b}{\lambda_1 d_1}}  ,  E_{k,1} \mid V_{\frac{ a\lambda_2^2 d_2^2 }{t_2^2}} \circ U_{\frac{t_1b}{\lambda_1 d_1}} \}.
\end{gathered} 
\end{equation}

The stage is now set for an application of \propref{eigenest} to the above equation. Namely for $i=1,2$, putting
\[ \ell_i =  \frac{t_ib}{\lambda_i d_i} , \q q_i = \frac{ a\lambda_i^2 d_i^2 }{t_i^2} \]
in the setting of \propref{eigenest} and using the estimete there in \eqref{ekek} and bounding absolutely, we get

\begin{equation} \nonumber
\begin{gathered}
 \{ E_{k, a\lambda_1^2 d_1^2} \mid U_{\frac{b}{\lambda_1 d_1}}, E_{k, a\lambda_2^2 d_2^2} \mid U_{\frac{b}{\lambda_2 d_2}}\}   \\
\ll  (a \lambda_1^2 d_1^2)^{-k+1} (a \lambda_2^2 d_2^2)^{-k+1} 
\underset{ \substack{ t_1^2 \mid a\lambda_1^2 d_1^2 \\t_2^2 \mid a\lambda_2^2 d_2^2 } } \sum \big( \frac{ a\lambda_1^2 d_1^2 }{t_1^2} \big)^{k-5/4} \big( \frac{ a\lambda_2^2 d_2^2 }{t_2^2} \big)^{k-5/4} \\
\ll (a \lambda_1^2 d_1^2)^{-1/4 +\varepsilon} (a \lambda_2^2 d_2^2)^{-1/4+ \varepsilon}. 
\end{gathered} 
\end{equation}

Putting this bound in \eqref{Ekmest}, we get
\begin{align*}
\{ \mc E_{k,m} , \mc E_{k,m} \} &\ll \underset{\lambda_1,\lambda_2}\sum (a \lambda_1^2 )^{\frac{k-1}{2}-\frac{1}{4} +\varepsilon} (a \lambda_2^2 )^{\frac{k-1}{2}-\frac{1}{4}+ \varepsilon}  \underset{d_1,d_2}\sum  (d_1)^{-1/2 +\varepsilon}  (d_2)^{-1/2 +\varepsilon} \\
&\ll \underset{\lambda_1,\lambda_2}\sum (a \lambda_1^2 )^{ \frac{k-1}{2}-\frac{1}{4} +\varepsilon} (a \lambda_2^2 )^{ \frac{k-1}{2}-\frac{1}{4}+ \varepsilon} \left( \frac{b}{\lambda_1} \right)^\varepsilon \left( \frac{b}{\lambda_2} \right)^\varepsilon\\
& \ll b^\varepsilon m^{k-\frac{3}{2} + \varepsilon} \ll m^{k-\frac{3}{2} + \varepsilon} .
\end{align*}

We state this as the main result of this section.
\begin{thm} \label{ekm-final}
For any given $\varepsilon>0$ one has
$\{ \mc E_{k,m} , \mc E_{k,m} \} \ll_{f,\varepsilon} m^{k-\frac{3}{2} + \varepsilon} .$
\end{thm}

\subsection{On estimates for extended norms for Fourier-Jacobi coefficients} \label{S2J}
We begin with a general estimate of the Petersson norm $\{\phi_N,\psi_N \}$ of Fourier Jacobi coefficients of two Siegel modular forms $F(Z)=\sum_{N \geq 0} \phi_N(\tau,z)  e^{2\pi i N\tau'}$ and $G(Z)=\sum_{N \geq 0} \psi_N(\tau,z) e^{2\pi i N\tau'}$ of degree two.

\begin{prop} \label{fj-final}
With the above notation and $N \geq 1$, we have
\[ \{\phi_N,\psi_N\}={\mathcal O}(N^k). \]
\end{prop}

\begin{proof}
 We recall from \lemref{phij-lem} that
\[ {\mathcal M}(\phi_N\overline{\psi_N}e^{-4\pi Nv'} \det(Y)^k)) \]
is bounded (uniformly in $N$) by a constant $C$.
On the other hand, we can write it explicitly as
\[ \left((2R(k-\frac{1}{2})\mc D^J+R^J_1)
(\phi_N\overline{\psi_N} v^k e^{-4\pi N \frac{y^2}{v}}) t^{k+1}
+(R^2\mc D^J+R^J_2)(\phi_N\overline{\psi_N} v^k e^{-4\pi N \frac{y^2}{v}}) t^{k+2}\right)
e^{-4\pi Nt} \]
where recall from \eqref{DJ0}, \eqref{DJ1}, \eqref{DJ2} and the disccusion thereafter, that
$R^J_1, R^J_2$ are the `inessential' parts of the Jacobi differential operators.
We put $t=y_4-\frac{y^2}{y_1}=\frac{1}{N}$
and recall that $R=-4\pi N$; 
then the expression above becomes
\[ A\cdot \mc D^J(\phi_N\overline{\psi_N} v^ke^{-4\pi N\frac{y^2}{y_1}})N^{-k}+ \ldots \]
with a nonzero constant $A= \left( 2(-4\pi)(k-\frac{1}{2}) + (4\pi)^2 \right)e^{-4 \pi} $ independent of $N$ and $\ldots$ denotes contributions from the
inessential operators $R_1^J$ and $R^J_2$.

Integration over the fundamental domain then kills the inessential part
and we get an estimate
\[ A\cdot \vert  \{\phi_N,\psi_N\}N^{-k}\vert \leq C\cdot \text{Volume}( \Gamma^J\backslash
{\mf H_1}\times {\mf C}). \qedhere \]
\end{proof}

\begin{rmk}
This is the same estimate as the one given by \cite{KS}
for Siegel cusp forms of degree $2$; this estimate holds more generally for
cusp forms of arbitrary degree.
\end{rmk}

\subsection{An estimate for $\phi^\circ_m$ of an arbitrary Siegel modular form} \label{phi0}

A Jacobi form admits a decomposition into a cusp form and
a sum of Jacobi Eisenstein series. 
For a degree $2$ Siegel modular form $F=\sum_m \phi_m(F) (\tau,z)e^{2\pi i m\tau'}$
we may therefore decompose the Fourier Jacobi coefficients for $m\geq 1$ as
\[ \phi_m(F)={\mathcal E}_{m}(F)+\phi_m^\circ(F) \]
where $\phi_m^\circ(F)$ is cuspidal and ${\mathcal E}_m(F)$ is in the space of Jacobi-
Eisenstein series. We drop the dependence on $F$ when its understood.
\begin{thm} \label{phi0-bd}
Let $F=\sum_m\phi_m(F)e^{2\pi i m\tau}$ be a 
Siegel modular form of degree 2 and weight $k$.
Then, with the above notation,
\[ \langle \phi^\circ_m(F),\phi^\circ_m(F) \rangle={O}(m^k). \]
\end{thm}

\begin{proof}
If $F$ is a Siegel Eisenstein series, then all the cuspidal
parts $\phi_m^\circ$ are zero \cite{BFJ}. If $F$ is cuspidal, then
$\phi_m=\phi_m^o$ and the claim can be found in \cite{KS}.
It remains the case of Klingen-Eisenstein series, i.e.
$F=E_{2,1}(f)$ for a degree one cusp form $f$.
In such cases, we may combine the two estimates
\[\{{\mathcal E}_m,{\mathcal E}_m\}={\mathcal O}(m^{k-\frac{3}{2}}),
\quad\quad
\{\phi_m,\phi_m\}={\mathcal O}(m^k) \]
from \thmref{ekm-final} and \propref{fj-final} and the fact that Jacobi-Eisenstein series are orthogonal to Jacobi cusp forms (w.r.t. $\{ \, , \}$ and $ \langle \, , \rangle$)
to get the result. The final claim is well-known for $\langle , \rangle$, and for $\{ \, ,\}$ we can invoke \corref{cusp-eis}.
\end{proof}

One of the main motivations of this work came from \cite{BV},
where the problem of estimating the Fourier coefficients of $\phi_m^o$
was raised with applications to quadratic forms in mind.
From the theorem above we get the following.

\begin{cor} \label{cmnr}
Let $F=\sum_m\phi_m(F) (\tau,z) e^{2\pi i m\tau'}$ be Siegel modular form
of degree $2$ and weight $k$, and put $\phi_m^\circ(\tau,z)=\sum_{n,r}
c^\circ_m(n,r)e^{2\pi (i n\tau+rz)}$; then the estimate
\[ c^\circ_m(n,r) ={O}\left( ( 4mn-r^2)^{\frac{k}{2}-\frac{1}{4}+\epsilon} \right) \]
holds, if the matrix $\left(\begin{smallmatrix} n &\frac{r}{2}\\
\frac{r}{2} & m\end{smallmatrix}\right)$ is reduced,  i.e. if $\mid r\mid \leq m\leq n$.
\end{cor}

\begin{proof} For the proof we just have to recall the general estimate
\[ c^\circ_m(n,r)\quad \ll_{k,\epsilon}\quad
\left(1+\frac{\det(T)^{\frac{1}{2}+\epsilon}}{m}\right)^{\frac{1}{2}}
\frac{\det(T)^{\frac{k}{2}-\frac{3}{4}}}{m^{\frac{k}{2}-1}}\cdot  \| \phi^\circ_m \|,
\]
which is actually valid for any Jacobi cusp form (\cite{Ko});
this gives the estimate claimed above, when combined with \thmref{phi0-bd}. 
Note that the Fourier coefficient
$c^\circ_m(n,r)$ is in general not invariant under the action 
of $\gltwo$ on the matrix $\left(\begin{smallmatrix} n &\frac{r}{2}\\
\frac{r}{2} & m\end{smallmatrix}\right)$.
\end{proof}

\begin{rmk} \label{detour}
The estimate for the Fourier coefficients of Siegel cusp
forms given in \cite{Ko} is just a little better than the one above
(the eponent of $4mn-r^2$ in \cite{Ko} is $\frac{k}{2}-\frac{13}{16}+\epsilon$); and this is due to the application of a result of Landau on Dirichlet series with non-negative Dirichlet coefficients, which is not available here; as we do not understand
whether the Dirichlet series
$\sum_{m \geq 1} \langle\phi_m^o,\phi^o_m \rangle m^{-s}$
(or equivalently the series $\sum_{m \geq 1} \{ \mc E_{k,m}, \mc E_{k,m} \} m^{-s}   $) has reasonable analytic properties.
\end{rmk}

\begin{rmk}[Some oscillation property]
For the case of a Klingen Eisenstein series $E_{2,1}(f)$ of degree 2
attached to a normalized Hecke eigen cusp form $f$ of degree one,
we computed the coefficients $c^\circ_m(n,r)$ explicitly in \cite{BV}:
\[ c^\circ_m(n,r)=c_{2,1} \frac{a_2^k(T)}{L_2(f,2k-2)}\,\frac{1}{2}
\sum_{0\not=u,v\in {\mf Z} } \frac{b(f, nu^2+ruv+mv^2)}{(nu^2
+ruv+mv^2)^{k-1}}. \]

Here $c_{2,1}$ is a constant, $L_2(f,2k-2)$ a value of the symmetric square
$L$-function attached to $f=\sum_{n=1} b(n,f)e^{2\pi i n\tau}$ and 
$a^k_2(T)$ is the $T$-Fourier coefficient of the degree $2$ Siegel Eisenstein 
series of weight $k$. In this formula, $m$ must be squarefree and
$T=\left(\begin{smallmatrix} n & \frac{r}{2}\\
\frac{r}{2} & m\end{smallmatrix}\right)$ has fundamental discriminant.
We use the well known asymptotic property $a_{2}^k(T) \asymp
\det(T)^{k-\frac{3}{2}}$ and obtain from the \corref{cmnr} that
\[  \sum_{0\not=u,v\in {\mf Z} } \frac{b(f, nu^2+ruv+mv^2)}{(nu^2
+ruv+mv^2)^{k-1}}={O}(\det(T)^{-\frac{k}{2}+\frac{5}{4} +\epsilon}) .\]

This is an improvement to the estimate obtained in \cite{BV}
for this series, using Deligne's bound for $ b(n,f) $.
There we got an estimate ${O}(\det(T)^{-\frac{k}{4}  +\frac{1}{4}+\epsilon})$.
The difference is explained by the oscillations in the Fourier coefficients of $f$.
\end{rmk}

\subsection{The growth of Fourier coefficients of (noncuspidal) Siegel 
modular forms of degree $2$}

A Siegel modular form $F=\sum_T a_F(T)e^{2\pi i trace(TZ)}$ of 
degree $2$ can be decomposed
as $F= a_F(0_2)\cdot E_{2,0} + E_{2,1}(f) + G $
where $E_{2,0}$ is the Siegel-Eisenstein series, $E_{2,1}(f)$
is a Klingen-Eisenstein series attached to an appropriate cusp form 
$f=\sum_n b(f,n)e^{2\pi i n\tau}$
of degree 1 and $G$ is a Siegel cusp form.
This gives a decomposition of Fourier coefficients in three parts.
\begin{equation} A_F(T)= A_F^0(T)+A_F(T)^1+A_G(T)
\label{decomposition}
\end{equation}
The ``middle term'', coming from $E_{2,1}(f)=\sum A_F(T)^1e^{2\pi i trace(TZ)}$
may further be decomposed by considering its Fourier-Jacobi expansion
$$E_{2,1}(f)(Z)=\sum_{m=0}^{\infty} \phi_m(\tau,z)e^{2\pi i m\tau'}$$
and decomposing $\phi_m$ for $m>0$
as $\phi_m= {\mathcal E}_{k,m} +\phi_m^0$, where
${\mathcal E}_{k,m}$ is a suitable Jacobi-Eisenstein series and
$\phi^0_m$ the cuspidal part of $\phi_m$.

Invoking the formula for $\mc E_{k,m}$ from \propref{e2E} we get the following expression for $A_F(T)$ with $T= \left( \begin{smallmatrix} n & r/2 \\ r/2 & m  \end{smallmatrix} \right) \in \Lambda_n^+$ :
\begin{equation} \label{asymp?}
  a_F(T) = a_F(0_2) a_2^k(T) + c^{-1}_k \ut{t^2 \mid m}\sum \alpha_m(t) a_2^k \left( \left( \begin{smallmatrix} n & r/2t \\ r/2t & m/t^2  \end{smallmatrix} \right) \right)  + c_\phi^\circ(n,r) + a_G(T).
\end{equation}
One can compute the cup form $f$ associated to $E^{2,1}(f)$ by the relation
\[  \Phi F = a_F(0_2) E_k +f,  \]
where $\Phi$ is the Siegel's $\Phi$-operator. Using this in \eqref{alphadef} gives us
\[  \alpha_m(t; f) = \ut{\ell \mid t}\sum \mu(\frac{l}{t}) \frac{g_{\Phi F} (m/\ell^2) }{g_k(m/\ell^2)} - c_k  a_F(0_2) \delta_{1,t} = \alpha_m(t; \Phi F) -c_k a_F(0_2) \delta_{1,t} , \]
where $\delta_{1,t}$ is the Kronecker delta function, which equals $1$ if $t=1$ and $0$ otherwise. Putting this back in \eqref{asymp?}, we get that
\begin{align}
a_F(T) &=  a_F(0_2)a_2^k(T) + c^{-1}_k \ut{t^2 \mid m}\sum \alpha_m(t; \Phi F) a_2^k \left( \left( \begin{smallmatrix} n & r/2t \\ r/2t & m/t^2  \end{smallmatrix} \right) \right) -   a_F(0_2)a_2^k(T) \nonumber \\
& \q \q + c_{\phi^\circ}(n,r) + a_G(T) \nonumber \\
&= c^{-1}_k \ut{t^2 \mid m}\sum \alpha_m(t; \Phi F) a_2^k \left( \left( \begin{smallmatrix} n & r/2t \\ r/2t & m/t^2  \end{smallmatrix} \right) \right)  + c_{\phi^\circ}(n,r) + a_G(T). 
\end{align}
We summarise this in the next theorem. 
\begin{thm} \label{genasy} Let the notation be as above. Then
\begin{equation} \label{gene}
 a_F(T) =c^{-1}_k \ut{t^2 \mid m}\sum \alpha_m(t; \Phi F) a_2^k \left( \left( \begin{smallmatrix} n & r/2t \\ r/2t & m/t^2  \end{smallmatrix} \right) \right)  + c_{\phi^\circ}(n,r) + a_G(T).
\end{equation} 
\end{thm}

\section{Some applications}

\subsection{An asymptotic formula for representation numbers }
In \cite{BV} we showed that a decomposition as in \eqref{decomposition}
becomes remarkably simple, if $F$ is a theta series.
Our estimates allow us to improve the
error term in \cite{BV} and to remove the conditions on the matrix $T$ and its minimum. We identify the binary quadratic form $Q(x,y)=nx^2 + rxy +my^2$ with the half-integral matrix $T = \left( \begin{smallmatrix} n & \frac{r}{2}\\
\frac{r}{2} & m\end{smallmatrix} \right)$. We let $A(S,T):=\#\{ X\in {\mf Z}^{2k,2}\, \mid \, \frac{1}{2} X^tSX=T\}$
and $A(S,m)=\#\{X\in {\mf Z}^{2k,1}\,\mid\, \frac{1}{2} X^tSX=m\}$
denote the respective representation numbers. Further, for $m \geq 1$ let $A^*(S,T)$ denote the number of primitive representations of $T$ by $S$.

Let us now recall Siegel's Hauptsatz for even positive unimodular lattices, which would be used in what follows. If $S=S_1, S_2,\ldots , S_h$ be inequivalent representatives of positive even unimodular lattices under the action of the unimodular group $\mrm{GL}(2k,\mf Z)$, define for $i \geq 1$ the theta series attached to $S$ defined as
\[ \theta^{(i)}(S)=\theta^{(i)}(S,Z) = \ut{X \in \mf {Z}^{(2k,i)}}\sum e^{ \pi i \mrm{tr} (X'SX \cdot Z)}, \]
which are elements of $M^i_k$. Then the Fourier expansion of such theta series reads
\[  \theta^{(i)}(S,Z) = \ut{ T \in \Lambda_i} \sum A(S,T) e^{ 2\pi i \mrm{tr} (TZ)} .\]

\begin{thm}[\cite{Fr}] \label{Siegel}
With the above notation,
\[  \ut{ 1 \leq \nu \leq h}\sum m_\nu  \theta^{(i)}(S_\nu) = E^{(i)}_k, \]
where $E^{(i)}_k$ is the Siegel Eisenstein series of degree $i$ and weight $k$; 
\[m_\nu = A(S_\nu, S_\nu)^{-1} (\ut{ 1 \leq \nu \leq h}\sum A(S_\nu, S_\nu)^{-1}  )^{-1}. \]
\end{thm}

We would need a lemma about the representation numbers $A(S,T)$ and consequently about certain sum of the Fourier coefficients $a_2^k(T)$. We write, as is customary, $T=\left( \begin{smallmatrix} n & \frac{r}{2}\\
\frac{r}{2} & m \end{smallmatrix} \right) \in \Lambda_2^+$.

\begin{lem} \label{saviour}
Let the notation be as above. Then,
\begin{equation} \label{save}
 \ut{t^2 \mid m, \, t \mid r}\sum \mu(t) A(S, \left( \begin{smallmatrix} n & \frac{r}{2t}\\
\frac{r}{2t} & \frac{m}{t^2} \end{smallmatrix} \right) )  = \# \{  X,Y \in \mf{Z}^{k,1} \vert \left( X,Y \right)'S \left(X,Y \right) = T, \, Y \text{ primitive} \}. 
\end{equation}
\end{lem}

\begin{proof}
 then we see that
\[  \left( \begin{matrix} X'SX & X'SY\\
Y'SX & Y'SY \end{matrix} \right) =  \left( \begin{matrix} n & r/2 \\
r/2 & m\end{matrix} \right) .\]
Considering the above equation according to the content of the vector $Y$, we get, after letting the right hand side of \eqref{save} as $A(S,T)^\sharp$ that
\[  \sum_{t^2 \mid m,\, t \mid r} A(S, \left( \begin{smallmatrix} n & r/2t \\ r/2t & m/t^2  \end{smallmatrix} \right)  )^\sharp = A(S,T). \]
The lemma follows from this by M\"obius inversion.
\end{proof}

We shall keep the following notation in the rest of this section.
\begin{equation} \label{M}
\begin{aligned}
  \mbb{M}=\mbb{M}(S,T): &= \ut{t^2 \mid m,\, t \mid r}\sum \alpha_m(t; \theta^1(S, \cdot )) a_2^k \left( \left( \begin{smallmatrix} n & r/2t \\ r/2t & m/t^2  \end{smallmatrix} \right) \right), \text{ and recall} \\
\alpha_m(t; \theta^1(S, \cdot )) &= \ut{\ell \mid t}\sum \mu(\frac{t}{\ell} )\frac{A^*(S,m/\ell^2)}{g_k(m/\ell^2)}. 
\end{aligned}
\end{equation}

\begin{thm}  \label{repno}
Let $S$ be an even, unimodular, positive definite
quadratic form in $2k$ variables and let
$T= \left( \begin{smallmatrix} n & \frac{r}{2}\\
\frac{r}{2} & m\end{smallmatrix} \right)$ be a reduced binary quadratic form.
Then given any $\epsilon>0$ and for all $k \geq 4$ the following asymptotic formula holds,
\begin{equation} \label{asymp-f}
A(S, T) = c^{-1}_k \mbb{M}(S,T)
+O(\det(T)^{\frac{k}{2}-\frac{1}{4}+\epsilon}),
\end{equation} \label{ineq}
where $c_k = 2/\zeta(1-k)$, $g_k(m)$ is as defined in \eqref{alphadef}. Moreover, one has
\begin{equation} \label{mst}
 \mbb{M}(S,T) \gg_k \frac{A(S, \min(T))\det(T)^{k-3/2}}{\sigma_{k-1}(\min(T)) \log( \min(T))} .
\end{equation} 
\end{thm}
Here the implied constants depend only on $k$ and $\epsilon$, $\min(T)$ denotes the minimum of $T$: $\min(T):= \{ x'Tx \, \vert \, x \in \mf{Z}^{2,1} \setminus \{0\}  \}$. The error term in the corresponding statement in \cite{BV} had $\frac{3k}{4}-\frac{5}{4}+\epsilon$ in the exponent, so the above is an improvement in this regard. We record an immediate corollary.

\begin{cor} \label{repno-nonzero}
Keep the notation as in \thmref{repno}. For all $T$ such that $A(S, \min(T)) \neq 0$ and $\epsilon>0$, one has the lower bound $A(S,T) \gg_{\epsilon, k} \det(T)^{k-3/2-\epsilon}$.
\end{cor}

\begin{proof}[Proof of the corollary]
The corollary follows from \thmref{repno} by noting that 

$(i)$ both sides of the inequality in the corollary are invariant under $\gltwo$, so that it is enough to prove it for $T$ reduced; and

$(ii)$ the ratio $A(S, \min(T)) / \sigma_{k-1}(\min(T))$, if non-zero, is bounded below by a constant depending only on $k$. 

This in turn follows if we write $\theta^{(1)}(S,\cdot) = E_k + g_k$ for some cusp form $g_k$ depending only on $k$ and $S$ and use Hecke'e bound for Fourier coefficients of $g_k$. 
\end{proof}

\begin{proof}[Proof of \thmref{repno}]
The proof follows from \eqref{gene}. First of all, the `error terms' coming from the cuspidal contribution in \eqref{gene} is atmost $O(\det(T)^{\frac{k}{2}-\frac{1}{4}+\epsilon})$, this follows from the estimates from \eqref{cmnr} (since $T$ is reduced) and \cite{Ko}:
\begin{eqnarray}
c_{\phi^\circ}(n,r) & \ll_k & \det(T)^{\frac{k}{2}-\frac{1}{4}}\\
a_G(T) & \ll_k & \det(T)^{\frac{k}{2}-\frac{13}{36}}
\end{eqnarray}
Let us note here that here $F= \theta^2(S,Z)$ so that $\Phi F = \theta^1(S,\tau)$.
Moreover, recalling the definition of $\alpha_m(t; \theta^1(S, \cdot))$ from \eqref{alphadef} one checks easily that it equals the expression given in the theorem. The remaining point is to check that \eqref{asymp-f} is actually an asymptotic formula if $\det(T)$ are sufficiently large.

For ease of notation, let us consider the `main term' $\mbb{M}$ (from \eqref{M}) and rewrite it as
\begin{equation} \label{Main}
 \mbb{M} = \ut{\ell^2 \mid m}\sum \frac{A^*(S,m/\ell^2)}{g_k(m/\ell^2)} \ut{t^2 \mid m/\ell^2, \, t \mid r/\ell}\sum \mu(t) a_2^k \left( \left( \begin{smallmatrix} n & r/2\ell t \\ r/2\ell t & m/\ell^2 t^2  \end{smallmatrix} \right) \right). 
\end{equation}  
As in the proof of \lemref{saviour}, let us put
\[  A(S,T)^\sharp=\ut{t^2 \mid m, \, t \mid r/\ell}\sum \mu(t) A(S, \left( \begin{smallmatrix} n & {r/2t}\\
{r/2t} & {m/t^2} \end{smallmatrix} \right) ) , \q     a_2^k(T)^\sharp = \ut{t^2 \mid m, \, t \mid r}\sum \mu(t) a_2^k \left( \left( \begin{smallmatrix} n & r/2 t \\ r/2 t & m/ t^2  \end{smallmatrix} \right) \right).\]

Now from \lemref{saviour} it follows trivially (by a counting argument) that
\begin{equation*} 
A(S,T)^\sharp \geq A^*(S,T) ,
\end{equation*}
where recall that $A(S,T)^*$ denotes the number of primitive representations of $T$ by $S$, i.e., 
\[  A^*(S,T) = \# \{ \mc X \in Z^{(2k,2)} \vert \, S[\mc X]=T, \, \mc X \text{ primitive} \}. \]

Summing the above inequality over the inequivalent unimodular matrices $S_1=S,S_2,\ldots,S_h$ and then using Siegel's Hauptsatz \thmref{Siegel} for $T$, we conclude that
\begin{equation}  \label{*}
 a_2^k(T)^\sharp \geq a_2^k(T)^*
\end{equation}
where $a_2^k(T)^*$ denotes the $T$-th `primitive' Fourier coefficient of $E^k_{2,0}$ defined e.g., in \cite{BR} by means of the formula
\[  a_2^k(T) = \ut { \substack{  G \in \gltwo \backslash M_2(\mf Z), \, \det(G)\neq 0 \\ T[G^{-1}] \in \Lambda_2^+  } }\sum a^k_2(T[G^{-1}])^* ,\]
which arises commonly in the theory. For the theta series above, the above two notions of `primitive' Fourier coefficients coincide.

For us, we only need an asymptotic formula for such a quantity; and indeed it is known \cite{Ki1} that for any $T \in \Lambda^+_2$ one has,
\[  a_2^k(T)^* \asymp_k \det(T)^{k - 3/2},   \]
with the implied constant depending only on $k$. Then \eqref{*} shows that
\begin{equation} \label{astbd}
A(S,T)^\sharp \gg_k \det(T)^{k - 3/2}.
\end{equation}

In all, we now have from \eqref{Main} using positivity of the quantities involved along with \eqref{astbd} with $T$ replaced by $\begin{psmallmatrix} n & r/2 \ell \\ r/2 \ell & m / \ell^2  \end{psmallmatrix} $ that
\[  \mbb{M} \gg_k  \ut{\ell^2 \mid m}\sum \frac{A^*(S,m/\ell^2)}{g_k(m/\ell^2)} \cdot \frac{\det(T)^{k-3/2}}{\ell^{2k-3}} \gg_k \frac{A^*(S,m)}{g_k(m)} \cdot \det(T)^{k-3/2}.\]

We may remove the primitivity condition from the above bound. In fact by using the elementary inequality
\[ \frac{a_1}{b_1} + \frac{a_2}{b_2} + \cdots + \frac{a_j}{b_j} \geq \frac{a_1+b_1 + \cdots+a_j}{b_1+b_2+\cdots + b_j}  \]
with $a_i \geq 0$ and $b_i>0$ for all $i$, we can infer
\[ \mbb{M} \gg_k  \frac{ \ut{\ell^2 \mid m} \sum  A^*(S,m/\ell^2) \det(T)^{k-3/2} } {   \ut{\ell^2 \mid m} \sum  g_k(m/\ell^2) \ell^{2k-3}  }  \gg_k \frac{A(S,m) \det(T)^{k-3/2} }{m^{k-1} \ut{t^2 \mid m}\sum t^{-1}}   \gg_k \frac{A(S,m) \det(T)^{k-3/2} }{m^{k-1} \log(m)},  \]
and the proof is complete because we can assume that $m= \min(T)$, as $T$ is reduced.
\end{proof}

\subsection{Determination by Fundamental Fourier coefficients} In \cite{saha}, a remarkable result was proved: all Siegel cusp forms $F$ of degree $2$ are determined by $a_F(T)$ such that $-\det(2T)$ is a fundamental discriminant. In the same paper, it was asked 
(cf. \cite[remark~2.6]{saha}) whether the same result is true for  all of $M^2_k$, and it was indicated that similarity of  growth properties for Fourier coefficients of Siegel-and Klingen-Eisenstein series makes the situation rather delicate. With our work, we can settle this issue without working with explicit description for the Fourier coefficients of Eisenstein series. But first, let us prove a lemma about an $\Omega$-result on the Fourier coefficients of a cusp form in $S_k$.

\begin{lem} \label{omega}
Let $f \in S_k$ be non-zero. Then
\[  \limsup_{m \to \infty, \, \, m \,\, \mrm{square-free}}  |a_f(m)|/m^{(k-1)/2} >0. \]
\end{lem}

\begin{proof}
The proof is a consequence of one of the results in \cite{AD}, more precisely see pp.28--29, proof of Proposition~5.9 in \cite{AD}. Let $\alpha_f(m) = a_f(m)/{m^{(k-1)/2}}$ be the normalised Fourier coefficients of $f$. Then it is proved in \cite{AD} that for some square-free integer $M$ (composed of primes $ p < 87$),
\[  \sum_{m \, \mrm{square-free}, (m,M)=1, m \leq X}  |\alpha_f(m)|^2 >B_f X, \]
for some constant $B_f>0$ depending only on $f$. Now it is clear from the above lower bound that the sequence $|\alpha_f(m)|^2$ with $m$ square-free must be bounded away from zero along some subsequence. This implies the lemma.
\end{proof}

\begin{rmk}
We do not know how to prove this result by modifying the proofs of several $\Omega$-type results available in the liteature, eg., \cite{ram, rankin} to the square-free setting.
\end{rmk}

\begin{prop} \label{saha-all}
Let $F \in M^2_k$ be non-zero and $k \geq 5$. Then there exist infinitely many $T \in \Lambda_2^+$ such that $a_F(T) \neq 0$ and $-\det(2T)$ is odd and square-free (thus in particular a fundamental discriminant).
\end{prop}

\begin{proof}
We can clearly assume that $T = \left( \begin{smallmatrix} n & r/2 \\ r/2 & m \end{smallmatrix} \right)$ is reduced. 
We write $F= a_F(0_2) E^k_{2,0} + E^k_{2,1}(f) + G$ ($f \in S^1_k, G \in S^2_k$). 
We invoke \thmref{genasy} with $-\det(2T)$ fundamental. Then in the summation in \eqref{gene}, only the term corresponding to $t=1$ survives:
\begin{equation} \label{saha-d}
 a_F(T) =c^{-1}_k \frac{g_{\Phi F}(m)}{g_k(m)} a_2^k (T) + c_{\phi^\circ_m}(n,r) + a_G(T),
\end{equation}
where $\phi^\circ_m$ denotes the cuspidal part of the $m$-th Fourier Jacobi coefficient $  \phi_m$ of $E^k_{2,1}(f) $. Notice from \eqref{saha-d} that it is enough to prove the result for those $F$ for which $\Phi F$ is cuspidal, i.e., we can assume that the Siegel Eisenstein part of $F$ is zero, i.e., $\Phi F=f$.

Moreover, since $T$ is reduced, we can apply the estimate \corref{cmnr} and use \cite{Ko} to conclude that 
\begin{equation} \label{upp}
|c_{\phi^\circ_m}(n,r)| + |a_G(T)| \ll_{F, \epsilon} \det(T)^{k/2 - 1/4 + \epsilon}.
\end{equation}

To prove what we want, first we shall prove that for any sequence $4<m_1<m_2<\cdots$ of integers there exists a sequence $1<D_1<D_2<\cdots$ such that $-D_j$ is a fundamental discriminant for all $j \geq 1$ and matrices $T_j \in \Lambda^+_2$ such that for all $j$, $T_j = \left( \begin{smallmatrix} * & *\\ * & m_j \end{smallmatrix} \right) $ is reduced and $\mrm{disc}(2T_j)= -D_j$. 

We can construct such matrices $T_j$ as follows. Our choice would be $T_1 = \left( \begin{smallmatrix} n_1 & 1/2\\ 1/2 & m_1 \end{smallmatrix} \right) $ with $n_1 \in \mf Z$ chosen such that $1 < m_1 < n_1$ (so that $T_1$ is reduced) and $D_1= \det (2T_1)= 4m_1n_1 -1$ is an odd prime (so that $-D_1$ a fundamental discriminant). By Dirichlet's theorem on primes in arithmetic progressions, the last condition is satisfied by infinitely many $n_1 \geq 1$, and we choose the smallest one greater than $m_1$. It is clear that we can proceed onwards, eg., for $m_2$, we choose by the same procedure such a prime $D_2$ which is also bigger than $D_1$ etc.

We can now complete the proof of the present proposition. From \lemref{omega} let $ \{ m_j \}$ be a sequence along which $|a_f(m_j)|/m_j^{(k-1)/2} > \eta$ for some $\eta>0$ and all $j\geq 1$. In \eqref{saha-d}, we let $T$ vary over the sequence $\{T_j\}$ constructed from the $\{ n_j,m_j \}$ as described above. First of all from \eqref{gkm} we get
\[   |g_f(m_j)| = |a_f(m_j)| > \eta \cdot m_j^{(k-1)/2}. \]

So (using the trivial bound $g_k(m) \ll m^{k-1}$ and that $D_j^{1/2} > m_j$) the first term in \eqref{saha-d} is bounded from the below (upto an absolute constant) by
\begin{align}
\frac{D_j^{k-3/2} } {m_j^{(k-1)/2} } \geq D_j^{3k/4 - 5/4} . \label{loww}
\end{align}
For $k \geq 5$, one has $D_j^{3k/4 - 5/4} \gg _\epsilon D_j^{k/2 - 1/4 + \epsilon}$. Thus the  proposition follows in this case (cf. \eqref{upp}) since $D_j \to \infty$ as $j \to \infty$. When $k=4$, $M^2_4 = \mathbb{C}E^{(2)}_4$ and thus we are fine in this case as well, as was observed at the beginning of the proof.
\end{proof}

\subsection{The Dirichlet series of Kohnen-Skoruppa with modified Petersson product} The expression ${\mathcal M}(F\cdot \overline{G} \cdot \det(Y)^k)$
attached to two Siegel modular forms of degree $2$ and weight $k$
defines function invariant under $\sptwo$ with rapid decay,
we may therefore integrate it against Eisenstein series and may unfold
the integral. For Siegel Eisenstein series this was done in \cite{BC}
for general degree. To do it for Eisenstein series
for other maximal parabolic subgroups one needs to know the relations
between growth killing operators for $\spnr$ on the one hand and
growth killing operators on Jacobi groups on the other. The work of Yamazaki \cite{Yam} could be useful here. We have exhibited the relation
in case of degree $2$ in the previous sections, the general case
needs further investigation.

For $Z\in {\mf H}_2$ and $s\in {\mf C}$ we define
an Eisenstein series 
\[ E_{2,1}(Z,s):=
\sum_{M\in C_{2,1}({\mf Z})
\backslash \sptwo } \left(\frac{\det \im(M \langle Z \rangle )}{ \im(M \langle Z \rangle _1}\right)^s. \]
This series is known \cite{KS} to converge for $\re (s)>2$, 
and furthermore the completed function 
\[ E_{2,1}^*(Z,s):=\pi^{-s}\Gamma(s)\zeta(2s)E_{2,1}(Z,s) \]
has meromorphic 
continuation to ${\mf C}$, the only singularities being first order 
poles in $s=2$ and $s=0$  with residues $1$ and $-1$ respectively. 
Moreover, it satisfies
a functional equation $E_{2,1}^*(Z,2-s)=E_{2,1}^*(Z,s)$. 

We unfold the Rankin-Selberg type integral (following the lines of \cite{KS}):
\[  RS:=\int_{ \sptwo \backslash {\mf H}_2} {\mathcal M}
\left(F(Z)\cdot\overline{G(Z)}\cdot \det(Y)^k\right)\cdot E_{2,1}(Z,s)
\frac{dXdY}{\det(Y)^3}
= \]
\[ \int_{C_{2,1}({\mf Z})\backslash {\mf H}_2} {\mathcal M}
\left(F(Z)\cdot \overline{G(Z)}\cdot\det(Y)^k\right)\cdot
\left(\frac{\det(Y)}{y_1}\right)^s
\frac{dXdY}{\det(Y)^3}. \]
A fundamental domain for $C_{2,1}({\mf Z})$ is given by
\[ \{\begin{psm} 
\tau & z\\
z & \tau'\end{psm} \in \mf H_2 \,\mid (\tau,z)\in 
\Gamma^J\backslash {\mf H_1}\times {\mf C}, v'> y^2 / v , \,
u'\bmod 1\}. \] 

We plug in the Fourier-Jacobi expansions of $F$ and $G$
and we may then integrate over $u' \mod 1$. We interchange
integration and summation over the index $N$.
For fixed $N$ we have then to consider
\begin{equation}\int_{\Gamma^J\backslash {\mf H_1}\times {\mf C}, \, v'>y^2/v}  \, {\mathcal M}
\left(\phi_N\overline{\psi_N}e^{-4\pi Ny_4} \det(Y)^k\right)\cdot
\left(\frac{\det(Y)}{y_1}\right)^s
\frac{dudxdY}{\det(Y)^3}.
\label{integral}\end{equation}
Denoting by $\ldots$ the contributions from the inessential operators, 
we may rewrite the integral, using
the variable $t=v'-\frac{y^2}{v}$ and the identity
\[  {\mathcal M}(\phi_N\cdot\overline{\psi_N}e^{-4\pi Nv'}\det(Y)^k)
= \]
\[ \left(-4\pi N(2k-1)t^{k+1}+(4\pi N)^2 t^{k+2}\right)\cdot 
\mc D^J(\phi_N\cdot\overline{\psi_N}\cdot v^k
e^{-4\pi N\frac{y^2}{v}}))e^{- 4 \pi Nt}+ \ldots \]
Integration over $t$ and then over 
$\Gamma^J\backslash {\mf H_1}\times{\mf C}$ gives two 
$\Gamma$-factors and we get for \eqref{integral}
\[ \left(-(2k-1) \frac{\Gamma(s+k-1)}{(4\pi N)^{s+k-2}}
+\frac{\Gamma(s+k)}{(4\pi N)^{s+k-2}}\right)\{\phi_N,\psi_N\}.
\]
We observe that
\[ \Gamma(s+k) -(2k-1)\Gamma(s+k-1)=(s-k)(k+s-2)\Gamma(s+k-2) \]
and we obtain
\begin{equation}
RS= (s-k)(s+k-2)\frac{\Gamma(s+k-2)}{(4\pi)^{s+k-2}}\sum_N 
\{\phi_N,\psi_N\}N^{-s-k+2}.
\end{equation}

Summarizing our discussion in this section let us now state the following proposition.
\begin{prop}
For Siegel modular forms $F,G$ of degree 2 and weight $k$
the Dirichlet series
\begin{equation} \label{dfg}
D(F,G,s):= \zeta(2s) \sum_{N} \{\phi_N,\psi_N\}N^{-s-k+2} 
\end{equation} 
has a meromorphic continuation to ${\mf C}$.
Moreover,
\[ D(F,G; s)^*:= (2 \pi)^{-2s}(s-k)(s+k-2)\Gamma(s+k-2)\Gamma(s)D(F,G,s) \]
satisfies a functional equation
\[ D(F,G; s)^*=D(F,G ; 2-s)^* \]
and it has simple poles at $s=2$ and $s=0$, the residue at $s=2$
beeing proportional to the modified Petersson product $\{F,G \}$ defined by 
means of $\mathcal M$.
\end{prop}

\begin{rmk}
It is a common feature, that polynomial factors like
$(s-k)(k+s-2)$ arise from differential operators. Note that this polynomial
is itself invariant under $s\longmapsto (2-s)$.
In the case of Siegel cusp forms we get back the results from \cite{KS}.
\end{rmk}
We give an example of the above theorem in the case of the Siegel Eisenstein series $E^k_{2,0}$. Let us recall that the Fourier Jacobi coefficients of $E^k_{2,0}$ were denoted by $e_{k,m}$ (cf. \eqref{siegel-fj}). Let $D(F,G; s)$ be defined as in \eqref{dfg}.

\begin{prop}
\[ D(E^k_{2,0}, E^k_{2,0} ; s) = Z(E^k_{2,0} , s +k-2)  \cdot \{E_{k,1}, E_{k,1} \},\]
where $Z(F,s)$ denotes the spinor $L$-function of an eigenform $F$ on $\sptwo$.
\end{prop}

\begin{proof}
We start with evaluating the inner product $\{ e_{k,N}, e_{k,N} \}$. Since $E^k_{2,0}$ is in the Maa{\ss} space of degree $2$, we know that $e_{k,N} = E_{k,1} \vert_{k,m} V_N$, where $E_{k,1}$ is the Eisenstein series in $J_{k,1}$.
\begin{equation} \label{ekm-inpo}
\{ e_{k,N}, e_{k,N} \} = \{  V_N^* V_N E_{k,1} , E_{k,1} \}
\end{equation}
From \cite{KS} we know that on the vector space $J^{cusp}_{k,1}$, one has the following expression for the map $V_N^* V_N$:
\[  V_N^* V_N= \sum_{t \mid N} \psi(t) t^{k-2} T^J(N/t),\]
where $V_N^*$ is the adjoint of $V_N$ on this space w.r.t. $\langle  ,  \rangle$, $\psi(t)$ is defined by the relation $\sum_{t \geq 1} \psi(t) t^{-s} = \zeta(s-1) \zeta(s) /\zeta(2s)$, and $T^J(n)$ denotes the $n$-th Hecke operator on $J_{k,1}$ (see \cite{EZ} for the definition). By section~\ref{VN}, the same relation holds on the space $J_{k,1}$ w.r.t. $\{ \, ,\}$ as well.

Therefore by the correspondence $J_{k,1} \longleftrightarrow  M^1_{2k-2}$ in \cite[Cor~3, Thm.~5.4]{EZ} which is compatible with Hecke operators, we can write
\begin{equation}
\{ e_{k,N}, e_{k,N} \} =  \sum_{t \mid N} \psi(t) t^{k-2} \sigma_{2k-3}(N/t) \cdot \{  E_{k,1} , E_{k,1} \}.
\end{equation}

Thus
\[ D(E^k_{2,0}, E^k_{2,0} ; s -k+2) = \]
\begin{align}
&= \zeta(2s -2k+4) \sum_{N=1}^\infty  \left( \sum_{t \mid N} \psi(t) t^{k-2} \sigma_{2k-3}(N/t) \right) N^{-s} \cdot \{  E_{k,1} , E_{k,1} \} \nonumber \\
&= \zeta(2s -2k+4) \sum_{t=1}^\infty \psi(t) t^{-s + k-2} \sum_{M=1}^\infty ( \sum_{d \mid M} d^{2k-3} ) M^{-s} \cdot \{  E_{k,1} , E_{k,1} \}  \nonumber \\
&=\zeta(2s -2k+4) \cdot  \frac{\zeta(s-k+1) \zeta(s-k+2) }{ \zeta(2s-2k+4)} \cdot \zeta(s) \zeta(s-2k+3) \cdot \{  E_{k,1} , E_{k,1} \}  \nonumber \\
&= \zeta(s) \zeta(s-k+1) \zeta(s-k+2) \zeta(s-2k+3) \cdot \{  E_{k,1} , E_{k,1} \} \label{Z} \\
&= Z(E^k_{2,0} , s)  \cdot \{E_{k,1}, E_{k,1} \}. \label{Zo}
\end{align}
The last equality is well-known, for instance it follows from the Zarkovskaya identity
\[  Z(F,s) = Z( \Phi(F) , s) Z(\Phi(F), s-k+2),\]
with $\Phi$ being the Siegel's $\Phi$-operator, see \cite{Fr}.
\end{proof}

\begin{rmk}
The result is in line with one of the results in \cite{KS} that if $F$ is an eigenform and $G$ is in the Maa{\ss} space, then $D(F,G;s)$ is proportional to $Z(F,s)$. It seems to be a rather difficult question to determine whether the quantity $\{E_{k,1}, E_{k,1} \}$ is non-zero.
\end{rmk}

{\small 
\begin{center}
{\bf \Large Appendix\\[0.2cm]
Some calculus of differential operators}
\end{center}

The purpose of this appendix is to prove the crucial 
identities (\ref{DJ0}), (\ref{DJ1}),({\ref{DJ2})
by proving such identies first for $\Phi_1$ and $\Phi_2$
and then summing up. In our approach, the simple form of
(\ref{DJ0}), (\ref{DJ1}),({\ref{DJ2}) drops off as a result
of lengthy computations. A more conceptual explanation is desirable.

The non-commutative ring of invariant differential operators for the 
(non-reductive)
Jacobi group $G^J({\mf R})$ is very complicated. Fortunately
Ochiai et al \cite{Ochiai} provided a set of generators for that ring,
exhibited below (with a slightly modified notation)
\begin{eqnarray*}
L_1&:=&-(z_1-\bar{z_1})^2
{\partial_1\bar{\partial_1}
}- 
(z_2-\bar{z_2})^2\di\dib \\
&&
-(z_1-\bar{z_1})(z_2-\bar{z_2})(\bar{\partial_1}\di+\partial_1\dib)
\\
L_2&:=& (z_1-\bar{z_1})\di\dib\\
L_3
&=&i(z_1-\bar{z_1})^2\left\{-\partial_1\dib\dib
+\bar{\partial_1}\di\di\right\}\\
&&+i(z_2-\bar{z_2})(z_1-\bar{z_1})\left(\di\di\dib-\di\dib\dib\right)\\
&&+2i\cdot  L_2\\
L_3'&:=&(z_1-\bar{z_1})^2\left\{-\partial_1\dib\dib
+\bar{\partial_1}\di\di\right\}\\
&&+(z_2-\bar{z_2})(z_1-\bar{z_1})\left(\di\di\dib-\di\dib\dib\right)\\
L_4&=&
\frac{1}{2}(z_1-\bar{z_1})^2\left(\partial_1+\bar{\partial_1}\right)\left(\di\di+\dib\dib\right)\\
&&-\frac{1}{2}(z_1-\bar{z_1})^2\left(\partial_1-\bar{\partial_1}\right)\left(\di\di-\dib\dib\right)\\
&&+(z_1-\bar{z_1})(z_2-\bar{z_2})\left(
\di+\dib\right)\di\dib
\end{eqnarray*} 

 \section{The calculus for $\Phi_1$ }
The aim here is to rewrite the Maass's operator 
$\Phi_1= \tr\left( (Z-\bar{Z})\cdot \{(Z-\bar{Z})\partial \bar{Z}\}^t \cdot 
\partial Z\right)$
as a differential operator for Jacobi forms.
More explicitly, this is
\[ (z_1-\bar{z_1})^2\frac{\partial^2}{\partial z_1\partial \bar{z_1}}
+(z_2-\bar{z_2})^2\frac{\partial^2}{\partial z_1\partial \bar{z_4}}
+(z_2-\bar{z_2})^2\frac{\partial^2}{\partial z_4\partial \bar{z_1}}
+(z_4-\bar{z_4})^2\frac{\partial^2}{\partial z_4\partial \bar{z_4}} \]
\[+(z_1-\bar{z_1})(z_2-\bar{z_2})\frac{\partial^2}{\partial z_1\partial 
\bar{z_2}}
+(z_1-\bar{z_1})(z_2-\bar{z_2})\frac{\partial^2}{\partial z_2\partial 
\bar{z_1}}\]
\[+(z_2-\bar{z_2})(z_4-\bar{z_4})\frac{\partial^2}{\partial z_2\partial 
\bar{z_4}}
+(z_2-\bar{z_2})(z_4-\bar{z_4})\frac{\partial^2}{\partial z_4\partial 
\bar{z_2}} \]
\[+\frac{1}{2}(z_1-\bar{z_1})(z_4-\bar{z_4})\frac{\partial^2}{\partial z_2
\partial \bar{z_2}}
+\frac{1}{2}(z_2-\bar{z_2})^2\frac{\partial^2}{\partial z_2
\partial \bar{z_2}}.\]

We start from a function $f$ on ${\mf H_1}\times {\mf C}$ 
and associate to it in the same way as in (\ref{h}), (\ref{H})
two functions $H$ and $h$ on ${\mf H}_2$ and 
${\mf H_1}\times {\mf C}$, related to each other by
$H(Z)= h(\tau,z)\cdot t^ke^{-4\pi Nt}$ with $t:= v-\frac{y^2}{v}$. 

We apply the differential operator $\Phi_1$ to $H$ and we have to express
$$\Phi_1(H)=\Phi_1\left(h\cdot (v'-\frac{y^2}{v})^k\cdot 
e^{-4\pi N(v'-\frac{y^2}{v}}\right)$$ 
by differential operators applied to $h$.
We define  
$t:= y_4-\frac{y_2^2}{y_1}$ and $R:= -4\pi N$.

We first consider the case $k=0$, which seems to be somewhat simpler.
Already here we get a sum of 29 terms, which (by a tedious elementary
calculation) can be recollected into Jacobi
differential operators, using the notation from above:
\[ \Phi_1(H)=
\left( -L_1+ i\cdot L_2 t +R t - R^2 
t^2\right)(h)\cdot e^{Rt}. \]

A formal k-fold differentiation w.r.t. R gives the expression for
$e^{-Rt}\cdot \Phi_1(H)$ in the case of 
arbitrary $k$ (to be then applied to the function $h$:
\[ -L_1 t^k+i\cdot L_2t^{k+1} +  kt^k+Rt^{k+1}
-k(k-1)t^k-2Rkt^{k+1}-R^2t^{k+2}
= \]
\[ -\left(L_1+k(k-2)\right)t^k+
\left(i\cdot L_2+ R-2Rk\right)t^{k+1}
- R^2t^{k+2} .\]

\section{The calculus for $\Phi_2$ }

As before, we consider functions $H$ and $h$; 
the second differential operator of Maass is defined by
\[ \Phi_2(H)= \det(Z-\bar{Z})^{\frac{5}{2}}\partial^{[2]}\overline{\partial}^{[2]}
\left(H\cdot \det(Z-\bar{Z})^{-\frac{1}{2}}\right)= \]

\[ 16\cdot\det(Y)^{\frac{5}{2}} \partial^{[2]}\bar{\partial}^{[2]}\left(
h
\cdot y_1^{-\frac{1}{2}}\cdot (y_4-\frac{y_2^2}{y_1})^{k-\frac{1}{2}}
\cdot e^{R(y_4-\frac{y_2^2}{y_1})}\right) \]

In a first step we determine (for $k=\frac{1}{2}$)
\[ 16\cdot\det(Y)^{\frac{5}{2}}\partial^{[2]}
\bar{\partial}^{[2]}\left(h\cdot y_1^{-\frac{1}{2}}
\cdot e^{R(y_4-\frac{y_2^2}{y_1})}\right) \]

We keep in mind that (writing $\partial_j$ for $\frac{\partial}{\partial z_j}$)
\begin{equation}
\partial^{[2]}\bar{\partial^{[2]}}=\partial_1\bar{\partial_1}\partial_4\bar{\partial_4}
-\partial_1\partial_4
\bar{\partial_2}^2-\bar{\partial_1}\bar{\partial_4}\partial_2^2+\partial_2^2
\bar{\partial_2^2}\label{Maass}\end{equation}
  
Again we get gives a lot of terms from $\Phi_2$, to be collected as follows
$$(A+B\cdot L_1+C\cdot L_2+D\cdot L_3)(h)\cdot t^ke^{Rt}$$
where
$A,B,C,D$ are polynomials of degree smaller or equal to 2 in $t$

To avoid  formulas, which are not necessary for us, 
we only consider the explicit form of
$A$ and $B$ here; $\frac{5}{2}$ should be viewed as
$2+k$ for $k=\frac{1}{2}$:
\[ A=0 \]
\[ B={R^2} L_1 t^{\frac{5}{2}} \]
To get weights $k+\frac{1}{2}$, we differentiate as before w.r.t. $R$
to get
\[ \left(R^2L_1 t^{\frac{5}{2}+k}+2Rkt^{\frac{5}{2}+k-1}+ k(k-1)t^{\frac{5}{2}
+k-2}+ \ldots \right)(h)e^{Rt} \] 
where $\ldots$ denotes sums of 
monomials involving $L_2$ and $L_3$. 

This is then true  not only for $k\in \frac{1}{2}+{\mf N}$, but also
 for arbitrary $s\in {\mf C}$, in particular
for $k'\in {\mf N}$ with $k':= k+\frac{1}{2}$ and we get

\[ \left( ( R^2 t^{2+k'}+2R(k'-\frac{1}{2})t^{1+k'}+ 
(k'-\frac{1}{2})
(k'-\frac{3}{2})t^{k'} )\cdot L_1
+ \ldots \right)(h)e^{Rt}. \]

}

\end{document}